\documentclass[12pt,a4paper,reqno]{amsart}
\usepackage{amssymb,latexsym,amsfonts,amsthm,upref,amsmath,capt-of}
\usepackage[english]{babel}
\usepackage[margin=1in]{geometry} 
\usepackage[foot]{amsaddr}
\usepackage[dvipsnames,x11names]{xcolor}
\definecolor{myblue}{HTML}{003388}
\definecolor{mygreen}{HTML}{338800}
\definecolor{myred}{HTML}{880033}
\usepackage{hyperref}   
\hypersetup{colorlinks,citecolor=mygreen,filecolor=black,linkcolor=myred,urlcolor=myblue}
\usepackage{comment} 
\usepackage{enumerate} 
\usepackage{tikz}
\usetikzlibrary{chains}
\usepackage{float}
\usepackage{cleveref}
\usepackage{mathtools}
\usepackage{cite}
\usepackage{filecontents}

\makeatletter

\newcommand{\leqnomode}{\tagsleft@true}
\newcommand{\reqnomode}{\tagsleft@false}

\makeatother
\newtheorem*{corintro*}{Corollary}
\newtheorem*{thm*}{Theorem}
\newtheorem*{lem*}{Lemma}
\newtheoremstyle{prim}{}{}{\normalfont}{}{\bfseries}{.}{ }{}
\newtheoremstyle{stil}{}{}{\slshape}{}{\bfseries}{.}{ }{}
\theoremstyle{stil}
\newtheorem{thm}{Theorem}[section]
\newtheoremstyle{defi}{}{}{}{}{\bfseries}{.}{ }{}
\theoremstyle{defi}

\theoremstyle{defi}
\newtheorem{rem}[thm]{Remark}
\theoremstyle{stil}
\newtheorem{pro}[thm]{Proposition}
\theoremstyle{stil}

\theoremstyle{stil}

\theoremstyle{prim}
\newenvironment{prf}{\noindent \textit{Proof.}}{\null\hfill$\qed$\hskip
2mm\vskip 2mm}
\newcommand{\g}{\mathfrak{g}}
\newcommand{\h}{\mathfrak{h}}

\newcommand{\nplus}{\mathfrak{n}_+}
\newcommand{\ntilde}{\widetilde{\mathfrak{n}}}
\newcommand{\htilde}{\widetilde{\mathfrak{h}}}
\newcommand{\gtilde}{\widetilde{\mathfrak{g}}}
\newcommand{\ntplus}{\widetilde{\mathfrak{n}}_+}

\newcommand{\gtgeq}{\widetilde{\mathfrak{g}}^{\geqslant 0}}
\newcommand{\gtle}{\widetilde{\mathfrak{g}}^{< 0}}
\newcommand{\ntplusgeq}{\widetilde{\mathfrak{n}}_+^{\geqslant 0}}
\newcommand{\ntplusleq}{\widetilde{\mathfrak{n}}_+^{< 0}}
\newcommand{\ot}{\otimes}
\newcommand{\CC}{\mathbb{C}}
\newcommand{\ZZ}{\mathbb{Z}}
\newcommand{\At}{A_{\theta}}
\newcommand{\NN}{\mathbb{N}}
\newcommand{\Cc}{\mathcal{C}}
\newcommand{\Ec}{\mathcal{E}}
\newcommand{\Dc}{\mathcal{D}}

\newcommand{\BN}{B_{N(k\Lambda_0)}}
\newcommand{\BL}{B_{L(k\Lambda_0)}}
\newcommand{\BV}{B_{V}}
\newcommand{\BNc}{\mathfrak{B}_{N(k\Lambda_0)}}
\newcommand{\BLc}{\mathfrak{B}_{L(k\Lambda_0)}}
\newcommand{\BNcbar}{\bar{\mathfrak{B}}_{N(k\Lambda_0)}}
\newcommand{\BLcbar}{\bar{\mathfrak{B}}_{L(k\Lambda_0)}}
\newcommand{\BVc}{\mathfrak{B}_{V}}
\newcommand{\BVcbar}{\bar{\mathfrak{B}}_{V}}
\newcommand{\fand}{\quad\text{and}\quad}
\newcommand{\Fand}{\qquad\text{and}\qquad}
\newcommand{\non}{\nonumber}
\newcommand{\beq}{\begin{equation}}
\newcommand{\eeq}{\end{equation}}
\newcommand{\f}{f_{L(k\Lambda_0)}}
\newcommand{\fN}{f_{N(k\Lambda_0)}}
\newcommand{\fV}{f_{V}}
\newcommand{\fVbar}{\bar{f}_{V}}

\newcommand{\W}{W_{L(k\Lambda_0)}}
\newcommand{\WN}{W_{N(k\Lambda_0)}}
\newcommand{\I}{I_{L(k\Lambda_0)}}
\newcommand{\IN}{I_{N(k\Lambda_0)}}
\newcommand{\IV}{I_{V}}
\newcommand{\vmax}{v_{L(k\Lambda_0)}}
\newcommand{\LLL}{L(k\Lambda_0)}
\newcommand{\NNN}{N(k\Lambda_0)}

\newcommand{\ndo}{\mathop{\mathrm{End}}}
\newcommand{\ch}{\mathop{\mathrm{ch}}}
\newcommand{\om}{\mathop{\mathrm{Hom}}}
\newcommand{\rez}{\mathop{\mathrm{Res}}}

\newcommand{\wvr}{\overline}
\newcommand{\wtld}{\widetilde}
\newcommand{\hhat}{\widehat{\mathfrak{h}}}
\newcommand{\hhle}{\widehat{\mathfrak{h}}^{< 0}}
\newcommand{\BLcc}{\mathfrak{B}_{L(\Lambda)}}
\newcommand{\BLccbar}{\bar{\mathfrak{B}}_{L(\Lambda)}}
\newcommand{\BLL}{B_{L(\Lambda)}}

\begin{document}

\title{Principal subspaces for the affine Lie algebras in types $D$, $E$ and $F$}

\author{Marijana Butorac$^1$}
\address{$^1$ Department of Mathematics, University of Rijeka, Radmile Matej\v{c}i\'{c} 2, 51\,000 Rijeka, Croatia}
\email{mbutorac@math.uniri.hr}

\author{Slaven Ko\v{z}i\'{c}$^2$}
\address{$^2$ Department of Mathematics, Faculty of Science, University of Zagreb,  Bijeni\v{c}ka cesta 30, 10\,000 Zagreb, Croatia}
\email{kslaven@math.hr}



\begin{abstract} 
We consider the principal subspaces of  certain level  $k\geqslant 1$ integrable highest weight modules and generalized Verma modules for the  untwisted affine Lie algebras in types $D$, $E$ and $F$.
Generalizing the approach of G. Georgiev we  construct their quasi-particle bases.
We use the  bases to derive    presentations of the principal subspaces, calculate their character formulae and find some new combinatorial identities.
\end{abstract}

\maketitle


\section{Introduction}\label{sec00}
\allowdisplaybreaks

Starting with   J. Lepowsky and S. Milne \cite{LM}, the fascinating connection between Rogers--Ramanujan-type identities  and affine Kac--Moody Lie algebras was extensively studied; see, e.g., \cite{LP,LW,MP,M} and references therein.
The principal subspaces of  standard modules, i.e. of integrable highest weight modules for the affine Lie algebras, introduced by B. L. Feigin and A. V. Stoyanovsky  \cite{FS}, present a remarkable example of this interplay between combinatorics and algebra.
In particular, their so-called quasi-particle bases 
 provide an interpretation  of the sum sides of various Rogers--Ramanujan-type identities; see \cite{B1,B2,B3,BS,FS,G,MiP}.
Aside from  quasi-particle bases,  numerous  research directions are focused on other  aspects of  principal subspaces and related structures such as 
certain generalized   principal subspaces \cite{AKS}, 
Feigin--Stoyanovsky's type subspaces \cite{BPT,JP,P},
realizations of  Jack symmetric functions \cite{CJ}, 
presentations of principal subspaces \cite{CLM1,CLM2,CLM3,CPS,PS,PS2,S1,S2}, 
Rogers--Ramanujan-type recursions \cite{CapLM1,CapLM2},
Koszul complexes \cite{Kan}, 
principal subspaces for quantum affine algebras and double Yangians \cite{c01,c05,c10}
 etc.  The key ingredient that all the aforementioned studies have in common is the application of   vertex-operator theoretic methods.
	
Let $\Lambda_0,\ldots ,\Lambda_l$ be the fundamental weights of the untwisted affine Lie algebra $\gtilde$ associated with the simple Lie algebra $\g$ of rank $l$.	
In this paper, we consider the principal subspaces $W_{\NNN}$ of the generalized Verma modules $\NNN$ and  the principal subspaces $W_{\LLL}$ of the standard modules $\LLL$ of highest weights $k\Lambda_0$ for  $\gtilde$ in types $D$, $E$ and $F$. The main result is a construction of the quasi-particle bases  $\BNc$ and $\BLc$ of the corresponding principal subspaces:

\makeatletter
\def\thmhead@plain#1#2#3{%
  \thmname{#1}\thmnumber{\@ifnotempty{#1}{ }\@upn{#2}}%
  \thmnote{ {\the\thm@notefont#3}}}
\let\thmhead\thmhead@plain
\makeatother

\begin{thm*}[\textbf{\ref{thm_baza}}]
{\slshape For any positive integer $k$ the set $\BVc$ forms a basis of the principal subspace $W_V$ of the $\gtilde$-module $V=N(k\Lambda_0),L(k\Lambda_0)$.}
\end{thm*}
The   bases are expressed in terms of monomials of certain operators, called quasi-particles, applied on the highest weight vector, whose charges and energies satisfy certain difference conditions.
Theorem \ref{thm_baza} for $\g$ of type $A_1$ goes back to   Feigin and Stoyanovsky  \cite{FS}. 
The $\g=A_l$   case was proved by G. Georgiev \cite{G} for all rectangular weights, i.e. for all integral dominant highest weights $\Lambda=k_0\Lambda_0+k_j\Lambda_j$.
The bases $\BVc$ for $\g=B_l, C_l, G_2$   were obtained by the first author  in  \cite{B1,B2,B3}. 
The $\g=A_l$  case for basic  modules   can be also recovered from   the recent result of K. Kawasetsu  \cite{Kawa}.
Our proof of Theorem \ref{thm_baza} in types $D$, $E$ and $F$ follows the approach in \cite{G} and relies on  \cite{B1,B2,JP}. 
In addition to  Theorem \ref{thm_baza}, in  Theorem \ref{thm_baza_DE} we construct quasi-particle bases of the principal subspaces $W_{L(\Lambda)}$ for all rectangular highest  weights $\Lambda$  in types $D$ and $E$, thus generalizing      \cite{G}.

Next, in Theorem \ref{thm_prezentacija}, we derive   presentations of the principal subspaces $W_{\LLL}$ for all types of $\g$.
The presentations of principal subspaces of standard $\gtilde$-modules $L(\Lambda)$ for the level $k$ integral dominant highest weights $\Lambda$ were established by
Feigin and Stoyanovsky \cite{FS} for $\g=A_1$ and $k= 1$.
Furthermore, the presentations were proved by
 C. Calinescu, Lepowsky and A. Milas  \cite{CLM1,CLM2,CLM3} 
for $\g=A_1$ and $k\geqslant 1$ and for $\g=A,D,E$ and $k=1$, and by C.  Sadowski   \cite{S1}  for $\g=A_2$ and $k\geqslant 1$. As explained in \cite{CLM1}, these a priori  proofs do not rely on the  detailed underlying structure, such as bases of the standard modules or of the principal subspaces.
Finally, Sadowski \cite{S2}  proved
 the   general case $\g=A_l$ for all $k\geqslant 1$   using  Georgiev's quasi-particle bases  \cite{G}.  
In contrast with \cite{CLM1,CLM2,CLM3,S1}, our proof   employs the sets $\BLc$ from Theorem \ref{thm_baza}, thus solving a simpler problem. 
In addition, using the quasi-particle bases from Theorem \ref{thm_baza_DE} we   obtain    presentations of the principal subspaces $W_{L(\Lambda)}$ for  all rectangular highest  weights $\Lambda$ in types $D$ and $E$; see Theorem \ref{thm_prezentacija_DE}. It is worth noting that, aside from the aforementioned cases covered in \cite{CLM1,CLM2,CLM3,S1}, the a priori proof  of  these presentations, which were originally conjectured in \cite{S2}, is still lacking. 

In the end, we use the bases from Theorems \ref{thm_baza} and \ref{thm_baza_DE} to explicitly write the character formulae for the principal subspaces. In particular, by regarding two different bases for $W_{\NNN}$ in types $D$, $E$ and $F$, we obtain three new families of combinatorial identities.

\section{Preliminaries}\label{sec10}
\numberwithin{equation}{section}

Let $\g$ be a complex simple Lie algebra of rank $l$
equipped with a nondegenerate invariant  symmetric bilinear  form $\langle \cdot,\cdot \rangle$ and let
  $\h $ be its Cartan subalgebra.
As the restriction of the form 	$\langle \cdot,\cdot \rangle$ on $\h$ is nondegenerate, it defines a symmetric bilinear  form on the dual $ \mathfrak{h}^{\ast}$. 
Let
$\Pi=\left\{\alpha_1,\ldots ,\alpha_l\right\}\subset\h^*$ be the basis of the root system $R$ of $\g$ with respect to $\h$ and let $x_\alpha\in\g$ with $\alpha\in R$ be the root vectors. The simple roots $\alpha_1,\ldots ,\alpha_l$ are labelled\footnote{\label{note1}
In contrast with \cite{H} and \cite[Table Fin]{K}, we reverse the labels in the Dynkin diagram of type $C_l$ in Figure \ref{figure}, so that the root lengths in the sequence  $\alpha_1,\ldots ,\alpha_l$   increase for all types of $\g$, thus getting a simpler formulation of Theorem \ref{thm_baza}.
} as in Figure \ref{figure}. 
We denote by $\alpha_1^{\vee},\ldots ,\alpha_l^{\vee}$ the corresponding simple coroots.
Let $\lambda_1,\ldots, \lambda_l\in\h^*$ be the  fundamental weights, i.e. the weights such that $\left<\lambda_i,\alpha_j\right>=\delta_{ij}$. Let $Q=\sum_{i=1}^l \ZZ\alpha_i$ and $P=\sum_{i=1}^l \ZZ\lambda_i$ be the root lattice and the weight lattice  of $\g$ respectively. 
We assume that the  form $\langle \cdot,\cdot \rangle$ is normalized so that $\langle \alpha,\alpha \rangle=2$ for every long root $\alpha\in R$. Hence, in particular, we have $\langle \alpha_i,\alpha_i \rangle\in \left\{2/3, 1, 2\right\}$ for all $i=1,\ldots ,l$.
Denote by $R_+$ and $R_{-}$    the sets of positive and negative roots. Let
$$\mathfrak{g} =\mathfrak{n}_{-}\oplus \mathfrak{h}\oplus \mathfrak{n}_{+},\qquad\text{where}\qquad \mathfrak{n}_{\pm }=\bigoplus_{\alpha \in R_{\pm}}\mathfrak{n}_{\alpha}\fand \mathfrak{n}_{\alpha} =\mathbb{C}x_{\alpha}\text{ for all }\alpha\in R,$$
be the triangular decomposition of $\g$; see \cite{H} for more details on simple Lie algebras.

\tikzset{node distance=1.8em, ch/.style={circle,draw,on chain,inner sep=2pt},chj/.style={ch,join},every path/.style={shorten >=5pt,shorten <=5pt},line width=1pt,baseline=-1ex}

\newcommand{\alabel}[1]{%
  \(\alpha_{\mathrlap{#1}}\)
}

\newcommand{\mlabel}[1]{%
  \(#1\)
}

\let\dlabel=\alabel
\let\ulabel=\mlabel

\newcommand{\dnode}[2][chj]{%
\node[#1,label={below:\dlabel{#2}}] {};
}

\newcommand{\dnodenj}[1]{%
\dnode[ch]{#1}
}

\newcommand{\dnodebr}[1]{%
\node[chj,label={right:\dlabel{#1}}] {};
}

\newcommand{\dydots}{%
\node[chj,draw=none,inner sep=1pt] {\dots};
}

\newcommand{\QRightarrow}{%
\begingroup
\tikzset{every path/.style={}}%
\tikz \draw (0,3pt) -- ++(1em,0) (0,1pt) -- ++(1em+1pt,0) (0,-1pt) -- ++(1em+1pt,0) (0,-3pt) -- ++(1em,0) (1em-1pt,5pt) to[out=-75,in=135] (1em+2pt,0) to[out=-135,in=75] (1em-1pt,-5pt);
\endgroup
}

\newcommand{\QLeftarrow}{%
\begingroup
\tikz
\draw[shorten >=0pt,shorten <=0pt] (0,3pt) -- ++(-1em,0) (0,1pt) -- ++(-1em-1pt,0) (0,-1pt) -- ++(-1em-1pt,0) (0,-3pt) -- ++(-1em,0) (-1em+1pt,5pt) to[out=-105,in=45] (-1em-2pt,0) to[out=-45,in=105] (-1em+1pt,-5pt);
\endgroup
}

\begin{align*} 
&A_l   &&\hspace{-15pt}
\begin{tikzpicture}[start chain]
\dnode{1}
\dnode{2}
\dydots
\dnode{l-1}
\dnode{l}
\end{tikzpicture}
&&B_l  &&\hspace{-15pt}
\begin{tikzpicture}[start chain]
\dnode{1}
\dnode{2}
\dydots
\dnode{l-1}
\dnodenj{l}
\path (chain-4) -- node[anchor=mid] {\(\Rightarrow\)} (chain-5);
\end{tikzpicture}
 \\\\
&C_l  &&\hspace{-15pt}
\begin{tikzpicture}[start chain]
\dnode{l}
\dnode{l-1}
\dydots
\dnode{2}
\dnodenj{1}
\path (chain-4) -- node[anchor=mid] {\(\Leftarrow\)} (chain-5);
\end{tikzpicture}
&&D_l  &&\hspace{-15pt}
\begin{tikzpicture}
\begin{scope}[start chain]
\dnode{1}
\dnode{2}
\node[chj,draw=none] {\dots};
\dnode{l-2}
\dnode{l-1}
\end{scope}
\begin{scope}[start chain=br going above]
\chainin(chain-4);
\dnodebr{l}
\end{scope}
\end{tikzpicture}
 \\\\
&E_6 &&\hspace{-15pt}
\begin{tikzpicture}
\begin{scope}[start chain]
\foreach \dyni in {1,...,5} {
\dnode{\dyni}
}
\end{scope}
\begin{scope}[start chain=br going above]
\chainin (chain-3);
\dnodebr{6}
\end{scope}
\end{tikzpicture}
&&E_7 &&\hspace{-15pt}
\begin{tikzpicture}
\begin{scope}[start chain]
\foreach \dyni in {1,...,6} {
\dnode{\dyni}
}
\end{scope}
\begin{scope}[start chain=br going above]
\chainin (chain-3);
\dnodebr{7}
\end{scope}
\end{tikzpicture}
 \\\\
&E_8 &&\hspace{-15pt}
\begin{tikzpicture}
\begin{scope}[start chain]
\foreach \dyni in {1,...,7} {
\dnode{\dyni}
}
\end{scope}
\begin{scope}[start chain=br going above]
\chainin (chain-5);
\dnodebr{8}
\end{scope}
\end{tikzpicture}
&&F_4 &&\hspace{-15pt}
\begin{tikzpicture}[start chain]
\dnode{1}
\dnode{2}
\dnodenj{3}
\dnode{4}
\path (chain-2) -- node[anchor=mid] {\(\Rightarrow\)} (chain-3);
\end{tikzpicture}
 \\\\
&G_2 &&\hspace{-15pt}
\begin{tikzpicture}[start chain]
\dnodenj{1}
\dnodenj{2}
\path (chain-1) -- node {\(\Rrightarrow\)} (chain-2);
\end{tikzpicture}
\end{align*}
\begingroup\vspace*{-\baselineskip}

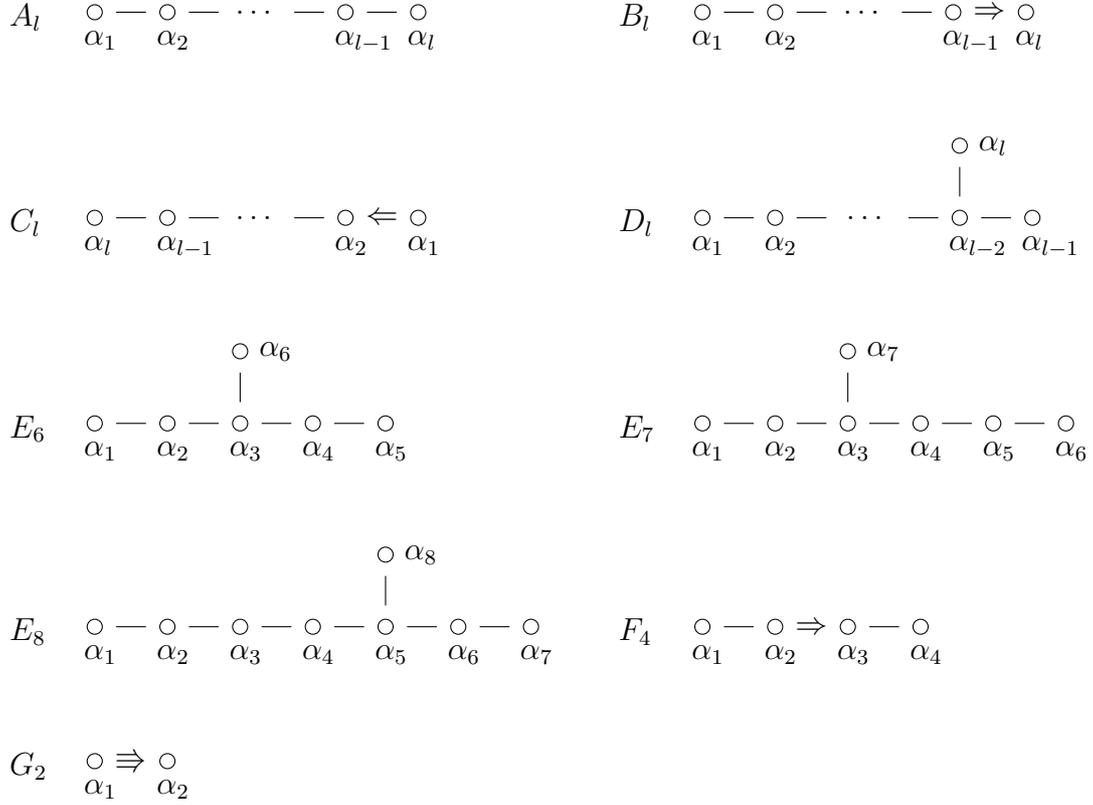
\captionof{figure}{Finite Dynkin diagrams}\label{figure}
\vspace*{\baselineskip}\endgroup

The  affine Kac--Moody Lie algebra $\gtilde$  associated to $\mathfrak{g}$ is defined by
$$\widetilde{ \mathfrak{g} }=\mathfrak{g}\otimes \mathbb{C}[t,t^{-1}]\oplus \mathbb{C}c\oplus \mathbb{C}d,$$
where the elements $x(m)=x\ot t^m$ for $x\in \g$ and $m\in\ZZ$ are subject to  relations
\begin{gather}
 \left[c,\widetilde{ \mathfrak{g}}\right]=0,\qquad \left[d,x(m)\right]=mx(m),\non\\
\left[x(m),y(n)\right]= \left[x, y\right](m+n) + \left\langle x, y \right\rangle m \delta_{m+n\,0}\, c.\label{defrel}
\end{gather}
We denote by $\alpha_0,\alpha_1,\ldots ,\alpha_l$ and $\alpha_0^\vee,\alpha_1^{\vee},\ldots ,\alpha_l^\vee$ the simple roots and the simple coroots of $\gtilde$. 
Let $\Lambda_i$ be the fundamental weights of $\gtilde$, i.e. the weights such that  $\Lambda_i (d)=0$ and $\Lambda_i (\alpha_j^{\vee})=\delta_{ij}$  for all $i,j=0,\ldots , l$. For more details on affine Lie algebras see \cite{K}.

Let $k_0,\ldots ,k_l$ be nonnegative integers such that $k=k_0+\ldots +k_l$ is positive and let $\lambda=k_1\lambda_1 +\ldots +k_l \lambda_l$. Denote by $U_{\lambda}$ the finite-dimensional irreducible $\g$-module of  highest weight $\lambda$.
The generalized Verma $\gtilde$-module $N(\Lambda)$   of  highest weight $\Lambda=k_0\Lambda_0+\ldots +k_l\Lambda_l$  and of level $k$ is defined as the induced $\gtilde$-module
$$N(\Lambda)= U(\gtilde)\ot_{U(\gtgeq)} U_{\lambda},$$
where  the action of the Lie algebra
$$\gtgeq=\bigoplus_{n\geqslant 0} (\g\ot t^n) \oplus\CC c\oplus\CC d $$
on $U_{\lambda}$ is given by
$$
\g\ot t^n \cdot u=0\text{ for all }n> 0,\quad c\cdot u=k u\fand d\cdot u=0\qquad\text{for all }u\in U_{\lambda}.
$$
Denote by  $L(\Lambda)$  the standard $\widetilde{\mathfrak{g}}$-module  of   highest weight $ \Lambda $ and of level $k$, i.e. the integrable highest weight $\widetilde{\mathfrak{g}}$-module which equals the unique simple quotient of the generalized Verma module $N(\Lambda)$. In particular, for $\lambda=0$ we obtain the
generalized Verma $\gtilde$-module $N(k\Lambda_0)$   of  highest weight $k\Lambda_0$ and  level $k=k_0$ which  possesses a vertex operator algebra structure.
Moreover,  $L(k\Lambda_0)$   is  a simple vertex operator algebra and the level $k$ standard $\widetilde{\mathfrak{g}}$-modules are  $L(k\Lambda_0)$-modules; see, e.g., \cite{LL,MP}.
Finally, recall that  
Poincar\'{e}--Birkhoff--Witt theorem for the universal enveloping algebra implies the vector space isomorphism
$$N(k\Lambda_0)\cong U(\gtle),\quad\text{where}\quad \gtle=\bigoplus_{n < 0} (\g\ot t^n) .$$
For more details on the  representation theory of affine Lie algebras see   \cite{K}.

\section{Quasi-particle bases of principal subspaces}\label{sec20}

In this section, we state  our main results, Theorems \ref{thm_baza} and   \ref{thm_baza_DE}. 

\subsection{Quasi-particles}\label{subsec21}

Introduce  the following subalgebras of $\gtilde$:
$$
 \ntplus=\mathfrak{n}_{+} \otimes \mathbb{C}[t,t^{-1}],\qquad 
 \ntplusgeq=\nplus \otimes \mathbb{C}[t]\Fand 
 \ntplusleq=\nplus \otimes t^{-1}\mathbb{C}[t^{-1}].
$$
Let $\Lambda$ be an arbitrary integral dominant weight of $\gtilde$.
Denote by $V$   the generalized Verma module $N( \Lambda )$ or the standard module $L(\Lambda )$ with a highest weight vector $v_V $.  Following Feigin and Stoyanovsky  \cite{FS}, we define the {\em principal subspace} $W_V$ of $V$ by
\begin{equation*} 
W_V=U(\ntplus)  v_{V}.
\end{equation*}

Consider the vertex operators
$$x_{\alpha_i}(z)=\sum_{m\in\ZZ} x_{\alpha_i}(m) z^{-m-1}\in \om(V,V((z)))\subset (\ndo V)[[z^{\pm 1}]], \quad  i=1,\ldots ,l.$$
Note that \eqref{defrel} implies $[x_{\alpha_i}(z_1),x_{\alpha_i}(z_2)]=0$ so that
\beq\label{quasi}
x_{n\alpha_i}(z)=\sum_{m\in\ZZ} x_{n\alpha_i}(m) z^{-m-n}=\underbrace{x_{\alpha_i}(z)\cdots x_{\alpha_i}(z)}_{n\text{ times}}=x_{\alpha_i}(z)^n
\eeq
is a well-defined element of $\om(V,V((z)))$ for all $n\geqslant 1$. In fact, $x_{n\alpha_i}(z)$ is the vertex operator associated with the vector $x_{\alpha_i}(-1)^n v_V \in V$.
As in \cite{G}, define the {\em quasi-particle of color $i$, charge $n$ and  energy $-m$} as the coefficient  $x_{n\alpha_i}(m)\in\ndo V$ of \eqref{quasi}. 

Consider the quasi-particle monomial
\beq\label{monomial}\tag{$m$}
b=
\left(x_{n_{r_{l}^{(1)},l}\alpha_{l}}(m_{r_{l}^{(1)},l})\ldots x_{n_{1,l}\alpha_{l}}(m_{1,l})\right)\ldots 
\left( x_{n_{r_{1}^{(1)},1}\alpha_{1}}(m_{r_{1}^{(1)},1})\ldots  x_{n_{1,1}\alpha_{1}}(m_{1,1})\right)
\eeq
in $\ndo V$.
Note that the quasi-particle colors in \eqref{monomial} are increasing from right to left and that the integers $r_j^{(1)}\geqslant 0$ with $j=1,\ldots ,l$  denote the parts of the conjugate partition of $n_j=n_{r_j^{(1)},j}+\cdots +n_{1,j}$; see \cite{G,B1,B2,B3} for more details. It is convenient to write  quasi-particle monomial \eqref{monomial} more briefly as 
\beq\label{briefly}
b=b_{\alpha_l}\,\cdots \,b_{\alpha_2}b_{\alpha_1}
,\quad\text{where}\quad
b_{\alpha_i}=x_{n_{r_{i}^{(1)},i}\alpha_{i}}(m_{r_{i}^{(1)},i})\ldots x_{n_{1,i}\alpha_{i}}(m_{1,i}) \text{ for } i=1,\ldots ,l.
\eeq

\subsection{Quasi-particle bases for \texorpdfstring{$\Lambda=k\Lambda_0$}{Lambda=kLambda0}}\label{subsec22}

Suppose that $\Lambda=k\Lambda_0$ for some positive integer $k$ so that $V$ denotes the generalized Verma module $N( k\Lambda_0 )$ or the standard module $L(k\Lambda_0 )$.
We   introduce certain difference conditions for  energies and charges of quasi-particles in \eqref{monomial}. First, for the adjacent quasi-particles of the same color we require that
\begin{align}
&\text{ for all}\quad i=1,\ldots ,l\fand p=1,\ldots, r_{i}^{(1)}-1\non\\
&\qquad  \text{if}\quad n_{p+1,i}=n_{p,i}\quad\text{then}\quad m_{p+1,i}\leqslant m_{p,i} -2 n_{p,i}.\label{d1}\tag{$c_1$}
\end{align}
Next, we turn to the difference conditions which describe the interaction of two quasi-particles of adjacent colors. For all $i=1,\ldots ,l$ define
\beq\label{niovi3}
\nu_i=\begin{cases}
1,&\text{if } \langle \alpha_i,\alpha_i \rangle =2,\\
2,&\text{if } \langle \alpha_i,\alpha_i \rangle =1,\\
3,&\text{if } \langle \alpha_i,\alpha_i \rangle =\frac{2}{3}
\end{cases}
\fand
i'=
\begin{cases}
l-2,&\text{if } i=l\text{ and }\g=D_l,\\
3,&\text{if } i=l\text{ and }\g=E_6,E_7,\\
5,&\text{if } i=l\text{ and }\g=E_8,\\
i-1,&\text{otherwise.} \\
\end{cases}
\eeq
Introduce the following difference conditions:
\begin{align}
&\text{for all}\quad i=1,\ldots ,l\fand p=1,\ldots, r_{i}^{(1)}\non\\
&\qquad  m_{p,i}\leqslant  -n_{p,i}
+ \sum_{q=1}^{r_{i'}^{(1)}}\min\left\{\textstyle\frac{\nu_{i} }{\nu_{i'}}n_{q,i'},n_{p,i}\right\}
- 2(p-1) n_{p,i} ,\label{d2}\tag{$c_2$}
\end{align}
where we set $r_0^{(1)}=0$ so that the  sum in \eqref{d2} is zero for $i=1$. In the end, we impose the following restrictions on the quasi-particle charges:
\beq\label{d3}\tag{$c_3$}
 n_{p,i}\leqslant k \nu_i \quad \text{for all}\quad i=1,\ldots ,l\fand p=1,\ldots, r_{i}^{(1)}.
\eeq

Let $\BN $ be the set of all monomials \eqref{monomial}, regarded as elements of $\ndo N(k\Lambda_0)$, satisfying conditions \eqref{d1} and \eqref{d2}. 
Moreover, let    $\BL $ be the set of all monomials \eqref{monomial}, regarded as elements of $\ndo L(k\Lambda_0)$, satisfying  \eqref{d1}, \eqref{d2} and \eqref{d3}. Finally, let
$$
\BVc=\left\{bv_V : b\in \BV\right\}\subset W_V  \quad\text{for}\quad V=N(k\Lambda_0),L(k\Lambda_0).
$$

\begin{thm}\label{thm_baza}
For any positive integer $k$ the set $\BVc$ forms a basis of the principal subspace $W_V$ of the $\gtilde$-module $V=N(k\Lambda_0),L(k\Lambda_0)$.
\end{thm}

Even though Theorem \ref{thm_baza} is formulated for an arbitrary untwisted affine Lie algebra $\gtilde$, we only give  proof for $\g$ of type $D$, $E$ and $F$; see Sections \ref{sec50} and   \ref{sec60}. The proofs for the remaining types can be found in \cite{B1,B2,B3,G}.

\subsection{Quasi-particle bases for rectangular weights in types \texorpdfstring{$D$}{D} and \texorpdfstring{$E$}{E}}\label{subsec23}
Suppose that the affine Lie algebra $\gtilde$ is of type $D_l^{(1)}$, $E_6^{(1)}$ or $E_7^{(1)}$. Let $\Lambda$ be the {\em rectangular weight}, i.e. the weight of the form
\beq \label{lambdaDE}
\Lambda = k_0\Lambda_0+k_j\Lambda_j,
\eeq
where $k_0,k_j$ are positive integers and $\Lambda_j$ is the fundamental weight of level one; cf. \cite{G}. Recall that  $j=1,l-1,l$ for $\gtilde=D_l^{(1)}$, $j=1, 6$ for $\gtilde=E_6^{(1)}$ and   $j=1$  for $\gtilde=E_7^{(1)}$; see \cite{K}. Denote by $k=k_0+k_j$   the level of $\Lambda$.
 Define
\beq \label{jde67}
j_{t}=\begin{cases}
0,&\text{if }\ \  1 \leqslant t \leqslant k_0,\\
j,& \text{if}\ \  k_0 < t \leqslant k_0+k_j.
\end{cases}
\eeq

Introduce the following difference condition:
\begin{align}
&\text{for all}\quad i=1,\ldots ,l\fand p=1,\ldots, r_{i}^{(1)}\non\\
&\qquad  m_{p,i}\leqslant  -n_{p,i}
+ \sum_{q=1}^{r_{i'}^{(1)}}\min\left\{\textstyle n_{q,i'},n_{p,i}\right\}
- 2(p-1) n_{p,i} -\sum_{t=1}^{n_{p,i}}\delta_{i j_t}.\label{d2DE67}\tag{$c'_2$}
\end{align}
Note that this condition differs from \eqref{d2} by a new term $\sum_{t=1}^{n_{p,i}}\delta_{i j_t}$.
For a given rectangular weight $\Lambda$   denote by  $\BLL $  be the set of all monomials \eqref{monomial}, regarded as elements of $\ndo L(\Lambda)$, satisfying  \eqref{d1}, \eqref{d2DE67} and \eqref{d3}. Finally, let
$$
\BLcc=\left\{bv_{L(\Lambda)}:b\in \BLL\right\}\subset W_{L(\Lambda)}.
$$

\begin{thm}\label{thm_baza_DE}
Let $\gtilde$ be the affine Lie algebra of type $D_l^{(1)}$, $E_6^{(1)}$ or $E_7^{(1)}$.
For any rectangular weight $\Lambda$  the set $\BLcc$ forms a basis of the principal subspace $W_{L(\Lambda)}$.
\end{thm}

 The proof of Theorem \ref{thm_baza_DE} is given in Section \ref{sec60}.

\section{Presentations of the principal subspaces \texorpdfstring{$\W$}{W-L}}\label{sec30}

In this section, we give the presentations of the principal subspaces $\W$ for an arbitrary untwisted affine Lie algebra $\gtilde$; see Theorem \ref{thm_prezentacija} below. Next, in Theorem \ref{thm_prezentacija_DE}, we give the presentations of $W_{L(\Lambda)}$ for all rectangular weights $\Lambda$ in types $D$ and $E$. 
As pointed out in Section \ref{sec00}, the presentations of the principal subspaces  of certain standard $\gtilde$-modules in types $A$, $D$ and $E$  were originally found and proved in \cite{FS,CLM1,CLM2,CLM3,S1,S2} while their general form was conjectured in \cite{S2}.

Let  $\Lambda$ be an integral dominant highest weight. Consider the natural surjective map
\begin{align}
f_{L(\Lambda)}\,\colon\, U(\ntplus)\,&\,\to\, W(\Lambda)\label{map}\\
a\,&\,\mapsto\, a\cdot v_{L(\Lambda)}.\non
\end{align}
For any  $i=1,\ldots ,l$ and integer $ m\geqslant k \nu_{i} +1$ define the elements $R_{\alpha_i} (-m)\in U(\ntplus)$ by
\begin{align*}
R_{\alpha_i} (-m)
=\sum_{\substack{m_1,\ldots ,m_{k\nu_{i}+1}\leqslant -1 \\
m_1+ \ldots + m_{k\nu_{i}+1}=-m }}
x_{\alpha_i}(m_1)\ldots x_{\alpha_i}(m_{k\nu_{i}+1}).
\end{align*}
Let $\I$ be the left ideal in the universal enveloping algebra $ U(\ntplus)$ defined by
\begin{align}\label{ideal}
\I\,=\,U(\ntplus)\,\ntplusgeq \, + \, \sum_{i=1}^l \sum_{m\geqslant k \nu_{i}+1}
 U(\ntplus)R_{\alpha_i} (-m). 
\end{align}
We have the following natural presentations of the principal subspaces:

\begin{thm}\label{thm_prezentacija}
For all positive integers $k$ we have
$$\ker\f =\I \qquad\text{or, equivalently,}\qquad \W\cong U(\ntplus)/\I.$$
\end{thm}

In Section \ref{sec50}, we  employ the sets $\BLc$ from Theorem  \ref{thm_baza} to     prove Theorem \ref{thm_prezentacija} for the affine Lie algebra $\gtilde =F_4^{(1)}$. 
We omit the proof for   other types of  $\gtilde$ since it goes analogously, by using the corresponding quasi-particle bases.

Let $\gtilde$ be the affine Lie algebra of type $D_l^{(1)}$, $E_6^{(1)}$ or $E_7^{(1)}$.
As in \cite{S2}, for a given rectangular weight $\Lambda=k_0\Lambda_0 + k_j\Lambda_j$
define the left ideal
 in the universal enveloping algebra $ U(\ntplus)$  by
\begin{align}\label{idealDE}
I_{L(\Lambda)}\,=\,I_{L((k_0+k_j)\Lambda_0)}\,+\,U(\ntplus)x_{\alpha_j}(-1)^{k_0 +1}. 
\end{align}

\begin{thm}\label{thm_prezentacija_DE}
Let $\gtilde$ be the affine Lie algebra of type $D_l^{(1)}$, $E_6^{(1)}$ or $E_7^{(1)}$.
For a given rectangular weight $\Lambda$    we have
$$\ker f_{L(\Lambda)} =I_{L(\Lambda)} \qquad\text{or, equivalently,}\qquad W_{L(\Lambda)}\cong U(\ntplus)/I_{L(\Lambda)}.$$
\end{thm}

The proof of Theorem \ref{thm_prezentacija_DE} is given in Section \ref{sec60}.

\begin{rem}\label{remarkLP} 
The form of the elements $R_{\alpha_i} (-m)$
is motivated by 
the   integrability condition
\beq\label{integrability}
x_{(k\nu_i +1) \alpha_i} (z)=0 \quad\text{on any level $k$ standard module},
\eeq
which is
due to Lepowsky and Primc \cite{LP}.
It implies quasi-particle charges
  constraint \eqref{d3}.
\end{rem}

\section{Proof of Theorems \ref{thm_baza} and \ref{thm_prezentacija} in type \texorpdfstring{$F$}{F}}\label{sec50}

In this section, we  prove  Theorems \ref{thm_baza} and \ref{thm_prezentacija} in type $F$. The proof is divided   into six steps, i.e.  Sections \ref{subsec51}--\ref{subsec56}.
We consider the affine Lie algebra $\gtilde$ of type $F_4^{(1)}$ so that  $l=4$ and   the basis $\Pi$ of the root system $R$ for the corresponding simple Lie algebra $\g$ consists of the simple roots
$\alpha_1,\alpha_2,\alpha_3,\alpha_4$;
see \cite[Chap. III]{H}.
The maximal root $\theta$ equals 
\beq\label{maxroot}
\theta=2\alpha_1 +3\alpha_2+4\alpha_3+2\alpha_4\quad\text{and satisfies} \quad \alpha_i(\theta^{\vee})=\delta_{1i}\text{ for }i=1,2,3,4.
\eeq

\subsection{Linear order on quasi-particle monomials}\label{subsec51}
In this section, we briefly cover some basic concepts originated in \cite{G} which are typically used to handle quasi-particle monomials. In particular, we introduce a certain linear order among such monomials which will come in useful   in Section \ref{subsec55}.
Let
\begin{align}
b=&
\Big(x_{n_{r_{4}^{(1)},4}\alpha_{4}}(m_{r_{4}^{(1)},4})\ldots x_{n_{1,4}\alpha_{4}}(m_{1,4})\Big)
\Big(x_{n_{r_{3}^{(1)},3}\alpha_{3}}(m_{r_{3}^{(1)},3})\ldots x_{n_{1,3}\alpha_{3}}(m_{1,3})\Big)\non
\\
&
\Big(x_{n_{r_{2}^{(1)},2}\alpha_{2}}(m_{r_{2}^{(1)},2})\ldots x_{n_{1,2}\alpha_{2}}(m_{1,2})\Big)
\Big( x_{n_{r_{1}^{(1)},1}\alpha_{1}}(m_{r_{1}^{(1)},1})\ldots  x_{n_{1,1}\alpha_{1}}\label{monomial4}\tag{$m_{F_4}$}(m_{1,1})\Big)
\end{align}
be an element of $\ndo V$, where $V=\NNN$ or $V=\LLL$, such that
\beq\label{extra}
n_{r_{i}^{(1)},i}\leqslant \ldots \leqslant n_{1,i}\fand 
m_{r_{i}^{(1)},i}\leqslant \ldots \leqslant m_{1,i}\qquad\text{for all }i=1,2,3,4.
\eeq
Define  the {\em charge-type} $\Cc$ and the {\em energy-type} $\Ec$ of $b$  by
\begin{align}
&\Cc=\left( n_{r_{4}^{(1)},4},\ldots , n_{1,4}; \,
n_{r_{3}^{(1)},3},\ldots , n_{1,3};\,
n_{r_{2}^{(1)},2},\ldots , n_{1,2};\,
 n_{r_{1}^{(1)},1},\ldots  , n_{1,1}\right),\label{charge-type}\\
&\Ec=\left( m_{r_{4}^{(1)},4},\ldots , m_{1,4}; \,
m_{r_{3}^{(1)},3},\ldots , m_{1,3};\,
m_{r_{2}^{(1)},2},\ldots , m_{1,2};\,
 m_{r_{1}^{(1)},1},\ldots  , m_{1,1}\right).\non
\end{align}
Moreover, define the {\em color-type} of $b$ as the quadruple $(n_{4},n_{3},n_{2},n_{1})$  such that $n_j$ denotes the sum of charges of all color $j$ quasi-particles, i.e. such that $n_j= n_{r_{j}^{(1)},j}+\ldots + n_{1,j}$.

Let $b_1,b_2$ be any two quasi-particle monomials of the same color-type, expressed as in \eqref{monomial4},  such that
their charges and energies satisfy \eqref{extra}. Denote their   charge-types  and    energy-types by $\Cc_1,\Cc_2$  and  $\Ec_1,\Ec_2$ respectively. 
Define the strict linear order among quasi-particle monomials of the same color-type by 
\beq\label{order1}
b_1< b_2\qquad \text{if}\qquad \Cc_1<\Cc_2\quad\text{or}\quad \Cc_1=\Cc_2 \text{ and } \Ec_1<\Ec_2,
\eeq
 where the order on  (finite) sequences of integers is defined as follows:
\begin{align}
&(x_p,\ldots ,x_1)< (y_r,\ldots ,y_1)\qquad   \text{if there exists }s\text{ such that }\quad\label{order2}\\
&x_1=y_1,\,\ldots,\,x_{s-1}=y_{s-1}\Fand s=p+1\leqslant r\quad \text{or}\quad x_s<y_s.\non
\end{align}

\subsection{Projection of the principal subspace} \label{subsec52} 
As in \cite{B1}, we now generalize Georgiev's projection \cite{G} to type $F$.
Consider  quasi-particle monomial \eqref{monomial4} as an element of  $\ndo\LLL$. Suppose that its charges and energies satisfy \eqref{extra}. Define its {\em dual charge-type} $\Dc$ as 
\begin{align}
&\Dc=\left(r^{(1)}_{4}, \ldots , r^{(2k)}_{4};\,
 r^{(1)}_{3}, \ldots , r^{(2k)}_{3};\,
 r^{(1)}_{2}, \ldots , r^{(k)}_{2}; \,
r^{(1)}_{1}, \ldots , r^{(k)}_{1}\right),\label{dual-charge-type}
\end{align}
where $r_{i}^{(n)}$ denotes the number of color $i$ quasi-particles of  charge greater than or equal to $n$  in  the monomial. Observe that, due to \eqref{integrability}, the monomial does  not posses any quasi-particles of color $i$ whose charge is strictly  greater than $k\nu_i$.

The standard module $L(k\Lambda_0)$ can be regarded as a submodule of the tensor product module
$L(\Lambda_0)^{\ot k}$ generated by the highest weight vector $
v_{L(k\Lambda_0)}=v_{L(\Lambda_{0})}^{\otimes k }$.
Let $\pi_{\Dc}$ be the projection of the principal subspace $\W$ on  the tensor product space
$$ 
W_{(\mu^{(k)}_{4};\mu^{(k)}_{3};r_{2}^{(k)};r_{1}^{(k)})}
\otimes \cdots \otimes  
W_{(\mu^{(1)}_{4};\mu^{(1)}_{3};r_{2}^{(1)};r_{1}^{(1)})}
\subset 
W_{L(\Lambda_0)}^{\otimes k}\subset L(\Lambda_0)^{\otimes k},
$$
where  
$W_{(\mu^{(t)}_{4};\mu^{(t)}_{3};r_{2}^{(t)};r_{1}^{(t)})}$
denote the $\mathfrak{h}$-weight subspaces of  the level $1$ principal subspace $W_{L(\Lambda_0)}$ of   weight $\mu^{(t)}_{4}\alpha_4+ \mu^{(t)}_{3}\alpha_3+r^{(t)}_{2}\alpha_2+ r_{1}^{(t)}\alpha_1 \in R$ with
\beq\label{miovi}
\mu^{(t)}_{i}=r^{(2t)}_{i}+ r^{(2t-1)}_{i}\qquad\text{for}\qquad t=1,\ldots ,k \fand i=3,4.
\eeq

We denote by the same symbol  $\pi_{\Dc}$ the generalization of the projection  to the space of formal series with coefficients in $W_{\LLL}$.
Applying the generating function corresponding to \eqref{monomial4} on the highest weight vector $v_{L(k\Lambda_0)}=v_{L(\Lambda_{0})}^{\otimes k }  $ we obtain
 \begin{align}
&\big(x_{n_{r_{4}^{(1)},4}\alpha_{4}}(z_{r_{4}^{(1)},4}) \cdots     x_{n_{1,4}\alpha_{4}}(z_{1,4})\big)
\big(x_{n_{r_{3}^{(1)},3}\alpha_{3}}(z_{r_{3}^{(1)},3}) \cdots     x_{n_{1,3}\alpha_{3}}(z_{1,3})\big)\non\\
&\times\big(x_{n_{r_{2}^{(1)},2}\alpha_{2}}(z_{r_{2}^{(1)},2})\cdots
 \cdots x_{n_{1,2}\alpha_{2}}(z_{1,2})\big)
\big(x_{n_{r_{1}^{(1)},1}\alpha_{1}}(z_{r_{1}^{(1)},1})\cdots  x_{n_{1,1}\alpha_{1}}(z_{1,1})\big)v_{L(k\Lambda_0)}.\label{eq:p1}
\end{align}
Relations \eqref{integrability} imply that by applying the projection $\pi_{\Dc}$ on \eqref{eq:p1} we get 
\begin{align} 
& \hspace{-6pt} \Big(x_{n_{r^{(2k-1)}_{4},4}^{(k)}\alpha_{4}}(z_{r_{4}^{(2k-1)},4})\cdots   x_{n_{1,4}^{(k)}\alpha_{4}}(z_{1,4})\Big)
\Big( x_{n_{r^{(2k-1)}_{3},3}^{(k)}\alpha_{3}}(z_{r_{3}^{(2k-1)},3})\cdots   x_{n_{1,3}^{(k)}\alpha_{3}}(z_{1,3})\Big)\non\\
&\hspace{-6pt}\times\Big(    x_{n_{r^{(k)}_{2},2}^{(k)}\alpha_{2}}(z_{r_{2}^{(1)},2})\cdots     x_{n{_{1,2}^{(k)}\alpha_{2}}}(z_{1,2})\Big) 
\Big( x_{n_{r^{(k)}_{1},1}^{(k)}\alpha_{1}}(z_{r_{1}^{(k)},1})\cdots     x_{n{_{1,1}^{(k)}\alpha_{1}}}(z_{1,1}) \Big)\ v_{L(\Lambda_{0})}\non\\
\ot\cdots \ot&\Big( x_{n_{r^{(1)}_{4},4}^{(1)}\alpha_{4}}(z_{r_{4}^{(1)},4})\cdots   x_{n_{1,4}^{(1)}\alpha_{4}}(z_{1,4})\Big)
\Big(x_{n_{r^{(1)}_{3},3}^{(1)}\alpha_{3}}(z_{r_{3}^{(1)},3})\cdots
\cdots   x_{n_{1,3}^{(1)}\alpha_{3}}(z_{1,3})\Big)\non\\
&\hspace{-6pt}\times \Big( x_{n_{r^{(1)}_{2},2}^{(1)}\alpha_{2}}(z_{r_{2}^{(1)},2})\cdots     x_{n{_{1,2}^{(1)}\alpha_{2}}}(z_{1,2})\Big)
\Big( x_{n_{r^{(1)}_{1},1}^{(1)}\alpha_{1}}(z_{r_{1}^{(1)},1})\cdots     x_{n{_{1,1}^{(1)}\alpha_{1}}}(z_{1,1}) \Big) v_{L(\Lambda_{0})}\label{find}
\end{align}
  multiplied by some  nonzero scalar, where we set  $x_{0\alpha_i}(z)= 1$. The integers $n_{p,i}^{(t)}$ in \eqref{find} are uniquely determined by 
$$
0 \leqslant  
 n^{(k)}_{p,i}\leqslant    \ldots \leqslant  n^{(2)}_{p,i} \leqslant  n^{(1)}_{p,i} \leqslant \nu_i
\fand 
n_{p,i}=\sum_{t=1}^k n^{(t)}_{p,i}\qquad\text{for all }  i=1,2,3,4
$$
and by the requirement that at most one $n_{p,i}^{(t)}$ equals $1$ when $i=3,4$.    Therefore, for every variable $z_{r,i}$, where $i=1,2,3,4$ and $r=1,\ldots , r_i^{(1)}$,
the projection $\pi_\Dc$ places at most one generating function $x_{\alpha_i}(z_{r,i})$ if   $i=1,2$   and at most two generating functions $x_{\alpha_i}(z_{r,i})$ if $i=3,4$   on each tensor factor of $W(\Lambda_0)^{\ot k}$.

\subsection{Operators \texorpdfstring{$A_\theta$}{A-theta} and \texorpdfstring{$e_\alpha$}{e-alpha}}\label{subsec53} 
Let $b\in \BL$ be a quasi-particle monomial of charge-type $\Cc$ and dual charge-type  $\Dc$. Denote the charges and the energies of its quasi-particles as in \eqref{monomial4}.
In this section, generalizing the approach  from \cite{B3}, we demonstrate how to reduce   $b$   to obtain a new  monomial  $b'\in \BL$ such that its  charge-type $\Cc'$ satisfies $\Cc'<\Cc$ with respect to linear order \eqref{order1}.   This will be a key step in the proof of linear independence of the set $\BLc$ in Section \ref{subsec54}.

Let $\At$ be the constant term of the operator 
$$x_{\theta}(z)=\sum_{r\in\ZZ} x_{\theta}(r) z^{-r-1}\in\ndo L(\Lambda_0)[[z^{\pm 1}]],$$ 
i.e. $\At = x_{\theta} (-1)$, where $\theta$ is the maximal root; recall \eqref{maxroot}.  
Consider the image of the vector $\pi_\Dc\, b\vmax\in\W\subset W_{L(\Lambda_0)}^{\ot k}$
 with respect to the
 operator
$$(\At)_s \coloneqq\underbrace{1\ot\cdots\ot 1}_{k-s}\ot  \At\ot \underbrace{1\ot\cdots \ot 1}_{s-1}\qquad\text{for}\qquad s=n_{1,1}. $$
This image can be obtained as the   coefficient of the variables
\beq\label{vars}
\wvr{z}\coloneqq z_{r_{4}^{(1)},4}^{-m_{r_{4}^{(1)},4}-n_{r_{4}^{(1)},4}}\cdots 
z_{2,1}^{-m_{2,1}-n_{2,1}}
z_{1,1}^{-m_{1,1}-n_{1,1}} 
\eeq
 in the expression
\beq\label{genfun}
(\At)_s \, \pi_\Dc\, x_{n_{r_{4}^{(1)},4}\alpha_{4}}(z_{r_{4}^{(1)},4}) \cdots 
x_{n_{2,1}\alpha_{1}}(z_{2,1})
x_{n_{1,1}\alpha_{1}}(z_{1,1}) v_{L(k\Lambda_0)}.
\eeq
Due to \cite{FHL}, the operator $\At$ commutes with the action of quasi-particles.  Hence, using \eqref{find}, we find that the $s$-th tensor factor (from the right) in \eqref{genfun} equals
 \begin{align*}
F_s=&\Big(x_{n_{r^{(2s-1)}_{4},4}^{(s)}\alpha_{4}}(z_{r_{4}^{(2s-1)},4})\cdots   x_{n_{1,4}^{(s)}\alpha_{4}}(z_{1,4})\Big)
\Big( x_{n_{r^{(2s-1)}_{3},3}^{(s)}\alpha_{3}}(z_{r_{3}^{(2s-1)},3})\cdots   x_{n_{1,3}^{(s)}\alpha_{3}}(z_{1,3})\Big)\\
&\times\Big(    x_{n_{r^{(s)}_{2},2}^{(s)}\alpha_{2}}(z_{r_{2}^{(s)},2})\cdots     x_{n{_{1,2}^{(s)}\alpha_{2}}}(z_{1,2})\Big) 
\Big( x_{n_{r^{(s)}_{1},1}^{(s)}\alpha_{1}}(z_{r_{1}^{(s)},1})\cdots     x_{n{_{1,1}^{(s)}\alpha_{1}}}(z_{1,1}) \Big) x_{\theta} (-1)v_{L(\Lambda_{0})}.\non
\end{align*}

Consider the Weyl group translation operator  $e_\alpha\in\ndo L(\Lambda_0)$ defined by
$$
 e_{\alpha}=\exp  x_{-\alpha}(1)\exp  (- x_{\alpha}(-1))\exp  x_{-\alpha}(1) \exp x_{\alpha}(0)\exp   (-x_{-\alpha}(0))\exp x_{\alpha}(0) $$
for $\alpha  \in R$; see \cite[Chap. 3]{K}. It possesses the following properties:
\begin{align}
& e_{\alpha}v_{L(\Lambda_0)}=-x_{\alpha}(-1)v_{L(\Lambda_0)}\quad\text{for every long root }\alpha,\label{eq:21}\\
&x_{\beta}(j)e_{\alpha}=e_{\alpha}x_{\beta}(j+\beta(\alpha\sp\vee)) \quad \text{for all }   \alpha,\beta \in R\text{ and } j \in \mathbb{Z}.\label{eq:22}
\end{align}
Using  \eqref{eq:21} and \eqref{eq:22} for $\alpha=\theta$ we rewrite the $s$-th tensor factor $F_s$ as  
\beq\label{efes}
F_s = - e_{\theta}\, F_s\, z_{r_{1}^{(s)},1}\cdots z_{2,1} z_{1,1}.
\eeq
Recall \eqref{maxroot} and notation \eqref{briefly}. 
Taking the  coefficient of variables \eqref{vars} in \eqref{efes}  we find
$$
(\At)_s \,\pi_{\Dc}\,bv_{L(k\Lambda_0)}
=-(e_{\theta})_s  \,\pi_{\Dc}\,b^{+}v_{L(k\Lambda_0)},
$$
where $(e_{\theta})_s $ denotes the action of $e_{\theta}$ on the $s$-th tensor factor (from the right) and
\begin{align*}
b^{+}
=b_{\alpha_{4}}\, b_{\alpha_{3}}\, b_{\alpha_{2}}\,b_{\alpha_{1}}^{<s}\,b_{\alpha_{1}}^{s},\quad\text{where}\quad
&b_{\alpha_{1}}^{<s}=x_{n_{r_{1}^{(1)},1}\alpha_{1}}(m_{r_{1}^{(1)},1})\cdots x_{n_{r_{1}^{(s)}+1,1}}(m_{r_{1}^{(s)}+1,1})\\  
\fand &b_{\alpha_{1}}^{s}= x_{n_{r_{1}^{(s)},1}}(m_{r_{1}^{(s)},1}+1)\cdots x_{n_{1,1}\alpha_{1}}(m_{1,1}+1).&
\end{align*}

Therefore, by applying the above procedure we increased the energies of all quasi-particles of color $1$ and charge $s=n_{1,1}$ in the monomial $b\in\BL$ by  $1$. We may   continue to apply the same procedure, now starting with $b^+ \vmax$, until we obtain the monomial
\begin{align*}
\wtld{b}
=b_{\alpha_{4}}\, b_{\alpha_{3}}\, b_{\alpha_{2}}\,\wtld{b}_{\alpha_{1}},\qquad\text{where}\qquad
\wtld{b}_{\alpha_{1}}= x_{n_{r_{1}^{(1)},1}}(\wtld{m}_{r_{1}^{(1)},1})\cdots x_{n_{1,1}\alpha_{1}}(\wtld{m}_{1,1})\qquad\text{and}\\
(\wtld{m}_{r_{1}^{(1)},1},\ldots ,\wtld{m}_{r_{1}^{(s)}+1,1}, \wtld{m}_{r_{1}^{(s)},1},\ldots,\wtld{m}_{1,1})
\hspace{-2pt}=\hspace{-2pt}(m_{r_{1}^{(1)},1},\ldots ,m_{r_{1}^{(s)}+1,1}, m_{r_{1}^{(s)},1}-m_{1,1}-s,\ldots,-s).
\end{align*}
Since $b$ is an element of $\BL$, the quasi-particle monomial   $\wtld{b}$ belongs to $\BL$ as well. Moreover, the charge-type and the dual charge-type of  $\wtld{b}$ equal $\Cc$ and $\Dc$ respectively.

By \eqref{eq:21} we have  $x_{\alpha_1}(-1)v_{L(\Lambda_0)}=-e_{\alpha_1}v_{L(\Lambda_0)}$.
Hence,  the vector $\pi_{\Dc}\,\wtld{b} v_{L(k\Lambda_0)}$, which belongs to $ \W\subset W_{L(\Lambda_0)}^{\ot k}$,  equals the coefficient of the variables
\beq\label{vars2}
\wvr{z}\,\Big(z_{r_{1}^{(s)},1}\cdots  z_{2,1} z_{1,1}\Big)^{m_{1,1}+s}
\eeq
 in 
\beq\label{genfun2}
 (-1)^s \,\pi_{\Dc}\, x_{n_{r_{4}^{(1)},4}\alpha_{4}}(z_{r_{4}^{(1)},4}) \cdots x_{n_{2,1}\alpha_{1}}(z_{2,1}) 
\, (1^{\ot (k-s)} \ot e_{\alpha_1}^{\ot s})
\,v_{L(\Lambda_0)}^{\ot k},
\eeq
where  $\wvr{z}$ is given by \eqref{vars}. 
We now employ \eqref{eq:22} to move  $1^{\ot (k-s)} \ot e_{\alpha_1}^{\ot s}$ 
all the way to the left in \eqref{genfun2}. 
Next, by dropping the invertible operator $(-1)^s(1^{\ot (k-s)} \ot e_{\alpha_1}^{\ot s})$ and taking the coefficient of variables \eqref{vars2} we get
$\pi_{\Dc'}\,b'\vmax$, where the quasi-particle monomial $b'$ of  charge-type $\Cc'$ and dual charge-type $\Dc'$ is given by
\begin{align*}
b'=b_{\alpha_{4}}\, b_{\alpha_{3}}\, b'_{\alpha_{2}}\,b'_{\alpha_{1}}
\qquad\text{for}\qquad
b'_{\alpha_{1}}=x_{n_{r^{(1)}_{1},1}\alpha_{1}}(\wtld{m}_{r^{(1)}_{1},1}+2n_{r^{(1)}_{1},1})\cdots   x_{n_{2,1}\alpha_{1}}(\wtld{m}_{2,1}+2n_{2,1}),&\\
b'_{\alpha_{2}}=x_{n_{r^{(1)}_{2},2}\alpha_{2}}(m_{r^{(1)}_{2},2}-n^{(1)}_{r_{2}^{(1)},2}-\cdots-n^{(s)}_{r_{2}^{(1)},2})\cdots x_{n_{1,2}\alpha_{2}}(m_{1,2}-n^{(1)}_{1,2}-\cdots-n^{(s)}_{1,2}).&
\end{align*}
Clearly,   the energies of the quasi-particles in colors $3$ and $4$ did not change.  Furthermore, if the dual charge-type $\Dc$ of $b$ equals
$$\Dc=\big(r^{(1)}_{4}, \ldots , r^{(2k)}_{4};\, r^{(1)}_{3}, \ldots , r^{(2k)}_{3}; \, r^{(1)}_{2}, \ldots , r^{(k)}_{2};\, r^{(1)}_{1},\ldots, r_1^{(n_{1,1})},\underbrace{0, \ldots, 0}_{k-s}\big),
$$
then the dual charge-type $\Dc'$  of $b'$  equals
$$\Dc'=\big(r^{(1)}_{4}, \ldots , r^{(2k)}_{4};\, r^{(1)}_{3}, \ldots , r^{(2k)}_{3};\, r^{(1)}_{2}, \ldots , r^{(k)}_{2};\, r^{(1)}_{1}-1,\ldots, r_1^{(n_{1,1})}-1,\underbrace{0, \ldots, 0}_{k-s}\big).
$$
In particular, we have $\Cc'<\Cc$ with respect to linear order \eqref{order1}.
Finally, by arguing  as in \cite[Proposition 3.3.1]{B2} one can check that $b'$ belongs to $\BL$.

\subsection{Linear independence of the sets \texorpdfstring{$\BVc$}{B-V}}\label{subsec54} 
In this section, we prove  linear independence of the set  $\BLc$. Linear independence of   $\BNc$ can be verified by arguing as in \cite[Sect. 3]{B1}.
Suppose there exists a linear dependence relation    among some elements   $b^a \vmax\in\BLc$, 
\begin{equation}\label{eq:d1}
\sum_{a \in A}
c_{a}\, b^a v_{L(k\Lambda_0)}=0, \quad\text{where}\quad c_a\in\CC,\, c_a\neq 0\text{ for all }a\in A
\end{equation}
and $A$ is a finite nonempty set. As the principal subspace $\W$  is a direct sum of its $\h$-weight subspaces, we can 
assume that all $b^a\in\BL$ posses the same color-type. 

Recall  strict linear order \eqref{order1} and  choose $a_0\in A$ such that $b^{a_0}< b^a$ for all $a\in A$, $a\neq a_0$. Suppose that the charge-type $\Cc$ and the dual charge-type $\Dc$ of $b^{a_0}$ are given by \eqref{charge-type} and \eqref{dual-charge-type} respectively. Applying the projection $\pi_{\Dc}$ on   \eqref{eq:d1} we obtain a linear combination of elements in
\begin{align*} 
&W_{(\mu^{(k)}_{4};\mu^{(k)}_{3};r_{2}^{(k)};0)}
\ot\cdots \ot 
 W_{(\mu^{(n_{1,1}+1)}_{4};\mu^{(n_{1,1}+1)}_{3};r_{2}^{(n_{1,1}+1)};0)} \\
&\ot W_{(\mu^{(n_{1,1})}_{4};\mu^{(n_{1,1})}_{3};r_{2}^{(n_{1,1})};r_{1}^{(n_{1,1})})}
\ot\cdots \ot 
W_{(\mu^{(1)}_{4};\mu^{(1)}_{3};r_{2}^{(1)};r_{1}^{(1)})},
\end{align*}  
recall \eqref{miovi}.
The definition of the projection $\pi_{\Dc}$ implies that all   $b^a v_{L(k\Lambda_0)}$ such that the charge-type of $b^a$ is strictly greater than $\mathcal{C}$  with respect to \eqref{order2} are
annihilated by $\pi_{\Dc}$. 
Therefore, we can assume   that all $b^a$ posses the same  charge-type $\Cc$ and, consequently, the same dual-charge-type $\Dc$.

As in \eqref{briefly}, write the monomials $b^a$ as $b^a=b^a_{\alpha_4}b^a_{\alpha_3}b^a_{\alpha_2}b^a_{\alpha_1}$, where $b^a_{\alpha_j}$ consist of quasi-particles of color $j$. We now apply the procedure described in Section \ref{subsec53} on
 the linear combination
\begin{equation}\label{eq:d2}
c_{a_0}\,\pi_{\Dc}\, b^{a_0}  v_{L(k\Lambda_0)}+
\sum_{a \in A,\,a\neq a_0}
c_{a}\,\pi_{\Dc}\, b^a v_{L(k\Lambda_0)}=0. 
\end{equation}
We repeat it until all  quasi-particles of color $1$ are removed from the  first summand $c_{a_0}\pi_{\Dc}\, b^{a_0}  v_{L(k\Lambda_0)}$.
This also removes all quasi-particles of color $1$ from other summands, so that \eqref{eq:d2}  becomes
\beq\label{thesum}
\wtld{c}_{a_0}\,\pi_{\wtld{\Dc}}\, b^{a_0}_{\alpha_4}b^{a_0}_{\alpha_3}\wtld{b}^{a_0}_{\alpha_2} v_{L(k\Lambda_0)} + \sum_{\substack{ a \in A,\,a\neq a_0\\b_{\alpha_1}^{a_0}=b_{\alpha_1}^{a} }}
\wtld{c}_{a}\,\pi_{\wtld{\Dc}}\, b^{a}_{\alpha_4}b^{a}_{\alpha_3}\wtld{b}^{a}_{\alpha_2} v_{L(k\Lambda_0)}=0 
\eeq
for some  quasi-particle  monomials $\wtld{b}^{a}_{\alpha_2}$ of color $2$ and  scalars $\wtld{c}_{a}\neq 0$ such that $\wtld{\Dc}$ is the dual charge-type of all quasi-particle monomials $b^{a}_{\alpha_4}b^{a}_{\alpha_3}\wtld{b}^{a}_{\alpha_2}$ in \eqref{thesum}. 
	The summation in \eqref{thesum}   goes over all $a\neq a_0$ such that $b^{a}_{\alpha_1}= b^{a_0}_{\alpha_1}$ because the summands $\pi_{\Dc}\, b^a v_{L(k\Lambda_0)}$ such that
  $b^{a_0}_{\alpha_1}< b^a_{\alpha_1}$ get annihilated in the process.

	The vectors $b^{a}_{\alpha_4}b^{a}_{\alpha_3}\wtld{b}^{a}_{\alpha_2} v_{L(k\Lambda_0)}$ in \eqref{thesum} belong to $\BLc$.
Furthermore, they can be realized as elements of the principal subspace of the level $k$ standard module $L(k\Lambda_0)$ with the highest weight vector $\vmax$ for the affine Lie algebra of type $C_3^{(1)}$. Moreover, their realizations belong to the corresponding basis  in type $C_3^{(1)}$, as given by Theorem \ref{thm_baza} (for a detailed proof in type $C_l^{(1)}$ see \cite{B2}). This implies $\wtld{c}_{a_0}=0$ and, consequently, $c_{a_0}=0$, thus contradicting \eqref{eq:d1}. Finally, we conclude that the set $\BLc$ is linearly independent.

\subsection{Small spanning sets \texorpdfstring{$\BVcbar$}{B-V}}\label{subsec55} 
In this section, we construct certain small spanning sets $\BNcbar$ and $\BLcbar$ for the quotients
$U(\ntplus)/ \IN  $ and $U(\ntplus)/ \I$ of the algebra $U(\ntplus)$ over its left ideals $\IN  =U(\ntplus) \ntplusgeq$ and $\I$ defined by  \eqref{ideal}. We denote by $\bar{x}$ the image of the element $x\in U(\ntplus)$  in these quotients with respect to the corresponding canonical epimorphisms.
First, we consider
$U(\ntplus)/ \IN$. 
By   Poincar\'{e}--Birkhoff--Witt theorem for the universal enveloping algebra we have
$$
U(\ntilde)=U(\ntilde_{\alpha_4})U(\ntilde_{\alpha_3})U(\ntilde_{\alpha_2})U(\ntilde_{\alpha_1}),\quad\text{where}\quad \ntilde_{\alpha_i}=\mathfrak{n}_{\alpha_i}\ot\CC[t,t^{-1}]\text{ and } \mathfrak{n}_{\alpha_i}=\CC x_{\alpha_i}.
$$
 By \eqref{defrel} quasi-particles of the same color commute, so all monomials
\beq\label{monomial4bar}\tag{$\bar{m}_{F_4}$}
\bar{b}=
\Big(\bar{x}_{n_{r_{4}^{(1)},4}\alpha_{4}}(m_{r_{4}^{(1)},4})\ldots \bar{x}_{n_{1,4}\alpha_{4}}(m_{1,4})\Big)\ldots 
\Big( \bar{x}_{n_{r_{1}^{(1)},1}\alpha_{1}}(m_{r_{1}^{(1)},1})\ldots  \bar{x}_{n_{1,1}\alpha_{1}}(m_{1,1})\Big)
\eeq
such that
their charges and energies satisfy 
\eqref{extra}
form a spanning set for   $U(\ntplus)/ \IN$. 

We now list two families of quasi-particle relations which can be used  to strengthen the    conditions in \eqref{extra}, i.e. to obtain a smaller spanning set. 
The first family is given for quasi-particles on $\NNN$ of color $i=1,2,3,4$ and charges $n_1$ and $n_2$ such that $n_2\leqslant n_1$:
\begin{align}
\left(\frac{d^p}{dz^p} x_{n_2\alpha_i} (z)\right)x_{n_1\alpha_i} (z)
=A_p(z) x_{(n_1 +1)\alpha_i} (z)
+B_p(z)  \frac{d^p}{dz^p} x_{(n_1+1)\alpha_i} (z), \label{r1}\tag{$r_1$}
\end{align}
where $p=0, 1,\ldots, 2 n_2 -1$ and $A_p(z), B_p(z)$ are some formal series with coefficients in the set of quasi-particle polynomials; 
 see  \cite{FS,G,JP}. 
As demonstrated in \cite[Remark 4.6]{JP}, see also  \cite[Lemma 2.2.1]{B1},
  relations \eqref{r1} can be used   to express $2n_2$ monomials of the form  
$$
x_{n_2\alpha_i}(m_2)x_{n_1\alpha_i}(m_1) ,\, \ldots \,  ,
 x_{n_2\alpha_i}(m_2-2n_2+1)x_{n_1\alpha_i}(m_1+2n_2-1)
$$
as a linear combination of monomials
$$
x_{n_2\alpha_i}(p_2)x_{n_1\alpha_i}(p_1) \quad \text{such that} \quad p_2 \leqslant m_2- 2n_2,\quad  p_1\geqslant m_1+2n_2\fand p_1 +p_2=m_1+m_2
$$
and monomials which contain a quasi-particle of color $i$ and charge $n_1+1$, thus possessing the greater charge-type. 
In particular, for $n_2 = n_1$ one can express $2n_2$ monomials 
$$
x_{n_2 \alpha_i}(m_2)x_{n_2 \alpha_i}(m_1)\quad\text{with} \ \ m_1-2n_2< m_2 \leqslant   m_1
$$
as a linear combination of monomials
$$
x_{n_2\alpha_i}(p_2)x_{n_2\alpha_i}(p_1)  \quad\text{such that} \quad p_2\leqslant p_1-2n_2$$
and monomials   which contain a quasi-particle of color $i$ and charge $n_2+1$, thus possessing the greater charge-type; cf. \cite[Corollary 2.2.2]{B1}. 

The second family of relations for quasi-particles on $\NNN$  is given by
\begin{align}
(z_{1}-z_{2})^{ M_{i}}x_{n_{i-1}\alpha_{i-1}}(z_{1})
x_{n_{i}\alpha_{i}}(z_{2})
=(z_{1}-z_{2})^{M_{i} }
x_{n_{i}\alpha_{i}}(z_{2})
x_{n_{i-1}\alpha_{i-1}}(z_{1}) \label{r2}\tag{$r_2$}\\
\text{for}\quad
i=2,3,4,\quad
M_{i}=\min \left\{\textstyle\frac{\nu_i}{\nu_{i-1}} n_{i-1},n_i\right\}\fand
n_{i-1},n_i\in\NN.\non
\end{align}
They follow by a direct computation employing the commutator formula for vertex operators; see, e.g., \cite[Chap. 6.2]{LL}. 
\begin{rem}
Due to \eqref{r2}, 
 the quasi-particles of colors $ 1$ and $ 2$ and the quasi-particles of colors $ 3$ and $ 4$ interact as the quasi-particles of colors $1$ and $2$ for the affine Lie algebra $A_2^{(1)}$ while the quasi-particles of colors $ 2$ and $ 3$ interact as the quasi-particles of colors $1$ and $2$ for the affine Lie algebra $B_2^{(1)}$.
\end{rem}

Let $\BNcbar$ be the set of all monomials \eqref{monomial4bar} satisfying difference  conditions \eqref{d1} and \eqref{d2} (with $l=4$ and
$i'=i-1$ for all $i=1,2,3,4$).
Using relations \eqref{r1} and \eqref{r2} and arguing as in \cite{B1,G} one can show that every monomial   of the form \eqref{monomial4bar} satisfying \eqref{extra} can be expressed as a linear combination of some monomials in $ \BNcbar$, so that $\BNcbar$ spans the quotient $U(\ntplus)/ \IN$.   The proof goes by induction on the  charge-type and total energy of quasi-particle monomials and relies on the properties of strict linear order \eqref{order1}.
Roughly speaking, difference condition \eqref{d1} follows from relations  \eqref{r1} for $n_2=n_1$; the last summand $- 2(p-1) n_{p,i}$ in \eqref{d2} follows from relations  \eqref{r1} for $n_2<n_1$; the  sum in \eqref{d2} follows from  \eqref{r2}. Finally, the first summand $-n_{p,i}$ in \eqref{d2} is due to
the fact that  each summand on the right hand side of
$$
x_{m\alpha_i}(n)=\sum_{n_1+\ldots +n_m =n} x_{\alpha_i}(n_1)\cdots x_{\alpha_i}(n_m),\quad\text{where }i=1,2,3,4\text{ and }n>-m,
$$
contains at least one quasi-particle $x_{\alpha_i}(n_j)$ with $n_j\geqslant 0$, so that $x_{m\alpha_i}(n)$
belongs to 
$\IN$ 
for $n>-m$.

We now consider
$U(\ntplus)/  \I$. 
It is clear that all monomials \eqref{monomial4bar}, regarded as elements of $U(\ntplus)/  \I$ and  satisfying difference conditions \eqref{d1} and \eqref{d2}, form a spanning set for the quotient  $U(\ntplus)/  \I$. 
However,  the form of the ideal $\I$, as defined in  \eqref{ideal}, implies additional relations
\beq\label{r3}\tag{$r_3$}
\bar{x}_{n\alpha_i}(z)=0\qquad\text{for all}\quad n\geqslant k\nu_i +1\fand i=1,2,3,4,
\eeq
in $U(\ntplus)/  \I$ which we now use  to  obtain a smaller spanning set; recall Remark   \ref{remarkLP}.

Suppose that monomial \eqref{monomial4bar}  satisfies difference conditions \eqref{d1} and \eqref{d2} and contains a quasi-particle $\bar{x}_{n_{j,i}\alpha_i }(m_{j,i})$
of charge $n_{j,i}\geqslant k\nu_i +1$ and color $1\leqslant i\leqslant 4$.
Clearly, such monomial 
coincides with the coefficient of the variables
\beq\label{vars3}
z_{r_{4}^{(1)},4}^{-m_{r_{4}^{(1)},4}-n_{r_{4}^{(1)},4}}\cdots z_{j,i}^{-m_{j,i}-n_{j,i}}
\cdots
z_{2,1}^{-m_{2,1}-n_{2,1}}
z_{1,1}^{-m_{1,1}-n_{1,1}}
\eeq
in the generating function
$$
  \bar{X}=\bar{x}_{n_{r_{4}^{(1)},4}\alpha_{4}}(z_{r_{4}^{(1)},4}) \cdots
	\bar{x}_{n_{j,i}\alpha_i} (z_{j,i})
	\cdots
	\bar{x}_{n_{2,1}\alpha_{1}}(z_{2,1}) 
	\bar{x}_{n_{1,1}\alpha_{1}}(z_{1,1})  .
$$
Introduce the Laurent polynomial
$$
P=\prod_{i=2}^4
\prod_{q=1}^{r_{i-1}^{(1)}}
\prod_{p=1}^{r_{i}^{(1)}}
\left(
1-\frac{z_{q,i-1}}{z_{p,i}}
\right)^{\min
\left\{
\textstyle \frac{\nu_i}{\nu_{i-1}} n_{q,i-1},n_{p,i}   
\right\}
}.
$$
By combining relations \eqref{r2} and \eqref{r3} we find $P\bar{X}=0$ as the operator
$\bar{x}_{n_{j,i}\alpha_i} (z_{j,i})$ in $P\bar{X}$ can be moved all the way to the right, thus annihilating the expression. By taking the coefficient of the variables \eqref{vars3} in $P\bar{X}=0$ we express \eqref{monomial4bar} as a linear combination
of some   quasi-particle monomials  of the same charge-type and of the same total energy $m_{r_4^{(1)},4} + \ldots + m_{1,1}$,  which are greater than 
\eqref{monomial4bar} with respect to  linear order \eqref{order1}.
However, there exists only finitely many such quasi-particle monomials which are nonzero.
Hence, by repeating the same procedure for an appropriate number of times, now starting with these new monomials, we find, after finitely many steps, that \eqref{monomial4bar}  equals zero. Therefore, we conclude that the set $\BLcbar$  of all monomials \eqref{monomial4bar} in  $U(\ntplus)/  \I$ which satisfy  difference conditions \eqref{d1}, \eqref{d2} and
\eqref{d3}
forms a spanning set for $U(\ntplus)/  \I$.

\subsection{Proof of Theorems \ref{thm_baza} and \ref{thm_prezentacija}}\label{subsec56} 

In Section \ref{subsec54}, we established the linear independence of the sets $\BNc$ and  $\BLc$. We now prove that they span the principal subspaces $\W$ and $\WN$, thus finishing the proof of Theorem \ref{thm_baza}. Moreover, as a consequence of the proof, we  obtain  the presentations of the principal subspace $\W$ given by Theorem \ref{thm_prezentacija}.
Introduce the 
natural surjective map
\begin{align*}
\fN\,\colon\, U(\ntplus)\,&\,\to\, W_{\NNN}\\
a\,&\,\mapsto\, a\cdot v_{\NNN},\non
\end{align*}
so that we can consider the cases $V=\LLL$ and $V=\NNN$ simultaneously. 
Recall that    
the surjective map $\f$ is given by \eqref{map},
the left ideal $ \I$ is defined by \eqref{ideal} and $\IN  =U(\ntplus) \ntplusgeq$ .

Let $V$ be $\NNN$ or $\LLL$. It is clear that the left ideal $\IV$ belongs to the kernel of $\fV$. Hence, there exists a unique map
\beq\label{mapbar}
\fVbar\colon U(\ntplus)/\IV \to W_V\quad\text{such that}\quad \fV=\fVbar\,\pi_V,
\eeq
where $\pi_V $ is the canonical epimorphism $U(\ntplus)\to  U(\ntplus)/\IV$. 
The map $\fVbar$ is surjective as $\fV$ is surjective and, furthermore,  it maps bijectively $\BVcbar$ to $\BVc$.
Therefore, the linearly independent set $\BVc$ spans the principal subspace $W_V$ and so it forms a basis of $W_V$, which proves    Theorem \ref{thm_baza}. This implies that the map \eqref{mapbar} is a vector space isomorphism, so, in particular, we conclude that $\ker \f = \I$, thus proving Theorem \ref{thm_prezentacija}.

\section{Proof of Theorems \ref{thm_baza}, \ref{thm_baza_DE} and \ref{thm_prezentacija_DE} in types \texorpdfstring{$D$}{D} and \texorpdfstring{$E$}{E} }\label{sec60}

In this section, unless stated otherwise, we denote by $\gtilde$ the affine Lie algebra of type $D_l^{(1)}$, $E_6^{(1)}$, $E_7^{(1)}$, $E_8^{(1)}$.
First, we give an outline of the proof of Theorem  \ref{thm_baza}   for $\gtilde$.   As the generalization of the arguments from Section \ref{subsec55}  is straightforward, we only discuss the proof of linear independence.
It  relies on the coefficients of certain level 1 intertwining operators 
 and on the vertex operator algebra construction   of  basic modules, thus resembling the corresponding proofs in types $A_l^{(1)}$, $B_l^{(1)}$ and $C_l^{(1)}$; see \cite{G,B1,B2}. 
In Section \ref{voacon} we recall the  aforementioned construction     while   in Section \ref{op}  we demonstrate  how to use the corresponding operators to complete the  proof of Theorem   \ref{thm_baza}. 
Next, in Section \ref{proof_DE} we add some details as compared to Sections \ref{sec50} and \ref{op} to take care of the modifications needed to carry out the argument for rectangular weights, i.e. to prove Theorems  \ref{thm_baza_DE} and \ref{thm_prezentacija_DE}.
Finally, in Section \ref{a} we construct different quasi-particle bases in type $E$, such that their linear independence can be verified by employing the operator $A_\theta$ associated with the   maximal root $\theta$, thus resembling the corresponding proof in type $F$ from Section \ref{sec50}.

\subsection{Vertex operator algebra construction   of  basic modules}\label{voacon}
We follow \cite{FLM,LL}
to review the vertex operator algebra construction   of the basic modules $L(\Lambda_i)$ \cite{FK,S}. 
Set 
$$\hhat_{*}=\bigoplus_{m \in \ZZ\setminus\{0\}}(\h \otimes t^m) \oplus \CC c\fand \hhle=\bigoplus_{m <0}\h \otimes t^m.$$
Let $M(1)=S(\hhle)$ be the Fock space for the Heisenberg algebra $\hhat_{*}$ with $h(-m)$ acting as multiplication and $h(m)$ acting  as differentiation on $M(1)$  for all $h \in \h$ and $ m\in \NN$. Consider the tensor products 
$$V_P = M(1)\otimes \CC\left[P\right]\fand V_Q = M(1)\otimes  \CC\left[Q\right], $$where $\CC\left[P\right]$ and $\CC\left[Q\right]$ denote the group algebras of the weight lattice $P$ and of the root lattice $Q$ with respective bases   $\{ e^{\lambda}: \lambda \in P\}$ and $\{ e^{\alpha}: \alpha \in Q\}$. We use the identification of group elements $e^{\lambda} = 1 \otimes e^{\lambda} \in V_P$.

Let $e_{\lambda}\colon V_P\to V_P$ be the linear isomorphism defined by
\beq\label{voaop3}
e_{\lambda}e^{\mu}= \epsilon(\lambda, \mu)e^{\mu+\lambda} \quad\text{for all }\lambda,\mu\in P,
\eeq 
where $\epsilon $ is a certain map $ P\times P \to \CC^{\times}$ satisfying $\epsilon(\lambda, 0)=\epsilon(0,\lambda)=1$ for all $\lambda\in P$; see \cite{FLM,LL} for  more details. 
The space $V_Q$ is equipped with a structure of a vertex operator algebra, with $V_P$  being a $V_Q$-module, by
\beq\non 
Y (e^{\lambda},  z) = E^{-}(-\lambda, z)E^{+}(-\lambda, z)e_{\lambda}z^{\lambda},\quad
\text{where} \quad E^{\pm}(-\lambda, z)=\exp \left( \sum_{n \leqslant 1}\lambda(\pm n)\frac{z^{\mp n}}{\pm n}\right)
 \eeq
 and   $z^{\lambda} = 1 \otimes z^{\lambda}$ acts by $ z^{\lambda}e^{\mu} = z^{\langle \lambda, \mu \rangle}$ for all $\lambda, \mu \in P$.
Moreover, the space $V_P$ acquires a structure of level one $\gtilde$-module 
via
$$x_{\alpha}(m)= \rez_z z^m Y (e^{\alpha}, z)\qquad \text{for }\alpha \in R\text{ and }m \in \ZZ.$$
With respect to this action, the space   $V_Q$ is identified with the standard module  $L(\Lambda_0)$ while the irreducible $V_Q$-submodules $V_Qe^{\lambda_i}$ of $V_P$  are identified with the standard  modules $ L(\Lambda_i)$ for all $i$ such that the weight $\Lambda_i$ is of level one.
The corresponding  highest weight vectors are  $v_{L(\Lambda_0)} =1$ and $v_{L(\Lambda_i)} = e^{\lambda_i}$.

\subsection{Operators \texorpdfstring{$A_{\lambda_i}$}{A-lambda} and   proof of Theorem \ref{thm_baza}}\label{op}
Let $b\in\BL$ be a quasi-particle monomial as in \eqref{monomial}, of charge-type $\Cc$ and dual charge-type  
$$
\Dc=\left(r^{(1)}_{l}, \ldots , r^{(k)}_{l};\, \ldots \,
 r^{(1)}_{2}, \ldots , r^{(k)}_{2}; \,
r^{(1)}_{1}, \ldots , r^{(k)}_{1}\right). 
$$
We now demonstrate how to carry out the procedure from Section \ref{subsec53}, i.e. how to reduce   $b$   to obtain a new  monomial  $b'\in \BL$ such that its  charge-type $\Cc'$ satisfies $\Cc'<\Cc$ with respect to linear order \eqref{order1}.
Denote by $I(\cdot, z)$ the intertwining operator of type $\binom{V_P}{V_P \,\,  V_Q}$,
$$
I(w, z)v =\exp (zL(-1))Y(v, -z)w, \ \ \text{where} \ \  w \in V_P, v \in V_Q,
$$
see \cite[Sect. 5.4]{FHL}. 
For $i=1,\ldots ,l$ let $A_{\lambda_i}$ be the constant term of $I(e^{\lambda_i}, z)$, that is  
\beq\non 
A_{\lambda_i}=\rez_z z^{-1}I(e^{\lambda_i}, z).
\eeq
We have
\beq \label{koefintvec}
A_{\lambda_i}v_{L(\Lambda_0)}=e^{\lambda_i}\quad\text{for all }i=1,\ldots ,l.
\eeq
In contrast with Section \ref{subsec53}, which relies on the application of the operators $A_{\theta}$ and $e_{\theta}$, we here make use of   $A_{\lambda_i}$ and $e_{\lambda_i}$ in a similar fashion. In particular, we employ the following property of   $e_{\lambda_i}$:
\beq \label{sv1}
e_{\lambda_i} x_{\alpha_j}(m)=(-1)^{\delta_{ij}}x_{\alpha_j}(m-\delta_{ij} )e_{\lambda_i}\quad\text{for all }i,j=1,\ldots ,l\text{ and }m \in \ZZ,
\eeq
 see \cite{CLM3} for more details. Moreover, we use the fact that the operators $A_{\lambda_i}$ commute with the action of $x_{\alpha}(z)$ for all $\alpha \in R$, which comes as a consequence of the commutator formula for  $x_{\alpha}(z)$ and $I(e^{\lambda_i}, z)$; see   \cite[Sect. 5.4]{FHL}.

As in Section \ref{subsec52}, denote by $\pi_{\Dc}$ the projection of the principal subspace $W_{L(k\Lambda_0)}$ on  
$$ 
W_{(r^{(k)}_{l};r^{(k)}_{l-1};\ldots;r_{2}^{(k)};r_{1}^{(k)})}
\otimes \cdots \otimes  
W_{(r^{(1)}_{l};r^{(1)}_{l-1};\ldots;r_{2}^{(1)};r_{1}^{(1)})}
\subset 
W_{L(\Lambda_0)}^{\otimes k}\subset L(\Lambda_0)^{\otimes k},
$$
where  
$W_{(r^{(t)}_{l};\ldots; r_{2}^{(t)};r_{1}^{(t)})}$
denote the $\mathfrak{h}$-weight subspaces of  the level $1$ principal subspace $W_{L(\Lambda_0)}$ of the weight $r^{(t)}_{l}\alpha_l+\cdots+r^{(t)}_{2}\alpha_2+ r_{1}^{(t)}\alpha_1 \in R$. 
Arguing as in Section \ref{subsec53}, we conclude that the image of $\pi_{\Dc}\, b v_{L(k\Lambda_0)}\in \W\subset W_{L(\Lambda_0)}^{\otimes k}$ with  respect to the operator
$$
(A_{\lambda_1})_s \coloneqq\underbrace{1\ot\cdots\ot 1}_{k-s}\ot  A_{\lambda_1} \ot \underbrace{1\ot\cdots \ot 1}_{s-1},\qquad\text{where}\qquad s=n_{1,1},
$$
equals the coefficient of the variables
\beq  \label{varsDE}
 z_{r_{l}^{(1)},l}^{-m_{r_{l}^{(1)},l}-n_{r_{l}^{(1)},l}}\cdots 
z_{2,1}^{-m_{2,1}-n_{2,1}}
z_{1,1}^{-m_{1,1}-n_{1,1}}
\eeq
 in the expression
\beq\label{513}
(A_{\lambda_1})_s \, \pi_\Dc\, x_{n_{r_{l}^{(1)},l}\alpha_{l}}(z_{r_{l}^{(1)},l}) \cdots 
x_{n_{2,1}\alpha_{1}}(z_{2,1})
x_{n_{1,1}\alpha_{1}}(z_{1,1}) v_{L(k\Lambda_0)}.
\eeq
Moreover, the $s$-th tensor factor in \eqref{513} (from the right) equals
$$
F_s=\Big(x_{n_{r^{(s)}_{l},l}^{(s)}\alpha_{l}}(z_{r_{l}^{(s)},l})\cdots   x_{n_{1,l}^{(s)}\alpha_{l}}(z_{1,l})\Big)
\cdots
\Big( x_{n_{r^{(s)}_{1},1}^{(s)}\alpha_{1}}(z_{r_{1}^{(s)},1})\cdots     x_{n{_{1,1}^{(s)}\alpha_{1}}}(z_{1,1}) \Big) e^{\lambda_1},
$$
where the integers $n_{p,i}^{(t)}$ are given by
$$
0 \leqslant  
 n^{(k)}_{p,i}\leqslant    \ldots \leqslant  n^{(2)}_{p,i} \leqslant  n^{(1)}_{p,i} \leqslant 1
\fand 
n_{p,i}=\sum_{t=1}^k n^{(t)}_{p,i}\qquad\text{for all }  i=1,\ldots, l.
$$

By combining \eqref{voaop3} and \eqref{sv1} we get 
\beq\label{opetDE}
F_s =(-1)^{r_{1}^{(s)}}e_{\lambda_1}\, F_s\, z_{r_{1}^{(s)},1}\cdots z_{2,1} z_{1,1}.
\eeq
Recall the notation from \eqref{briefly}. By taking the coefficient of variables  \eqref{varsDE} in \eqref{opetDE} we have
$$
(A_{\lambda_1})_s\pi_{\Dc}\,bv_{L(k\Lambda_0)}
=(-1)^{r_{1}^{(s)}}(e_{\lambda_1})_s  \,\pi_{\Dc}\,b^{+}v_{L(k\Lambda_0)},
$$
where $(e_{\lambda_1})_s $ denotes the action of $e_{\lambda_1}$ on the $s$-th tensor factor (from the right) and  
\begin{align*}
b^{+}
=b_{\alpha_{l}}\, \cdots\, b_{\alpha_{2}}b_{\alpha_{1}}^{<s}\,b_{\alpha_{1}}^{s} \quad\text{with}\quad
&b_{\alpha_{1}}^{<s}=x_{n_{r_{1}^{(1)},1}\alpha_{1}}(m_{r_{1}^{(1)},1})\cdots x_{n_{r_{1}^{(s)}+1,1}}(m_{r_{1}^{(s)}+1,1})\\  
\quad\text{and }\quad &b_{\alpha_{1}}^{s}= x_{n_{r_{1}^{(s)},1}}(m_{r_{1}^{(s)},1}+1)\cdots x_{n_{1,1}\alpha_{1}}(m_{1,1}+1).&
\end{align*}
Note that the   monomial $b^+$ belongs to $\BL$.  

 As in Section \ref{subsec53}, we can now continue to apply this procedure until we obtain a  monomial  $b'\in \BL$ of  charge-type   $\Cc'<\Cc$. Finally, by repeating the arguments from Section \ref{subsec54} almost verbatim, we can prove the linear independence of the set $\BLc$. However, in contrast with  Section \ref{subsec54}, where the quasi-particle basis in type $F_4^{(1)}$ was reduced to a basis in type $C_3^{(1)}$,   the quasi-particle basis in type $D_l^{(1)}$, $E_6^{(1)}$, $E_7^{(1)}$ or $E_8^{(1)}$ is reduced, after sufficient number of steps,  to a basis in type $A_1^{(1)}$ from Theorem \ref{thm_baza}. Note that such a modification of the argument is possible  because we have the operators $A_{\lambda_i}$ and $e^{\lambda_i}$ satisfying \eqref{koefintvec} and \eqref{sv1} at our disposal; cf. corresponding properties
\eqref{eq:21} and \eqref{eq:22} for $\alpha=\theta$.

\subsection{Proof of Theorems \ref{thm_baza_DE} and \ref{thm_prezentacija_DE}} \label{proof_DE}

Let $\gtilde$ be the affine Lie algebra of type $D_l^{(1)}$, $E_6^{(1)}$ or $E_7^{(1)}$ and let $\Lambda=k_0\Lambda_0 + k_j\Lambda_j$ be an arbitrary rectangular weight, as  defined in Section \ref{subsec23}.
First, we prove that the set $\BLcc$ is linearly independent.
As in Section \ref{subsec52}, we  regard the standard module $L(\Lambda)$ as the submodule of $L(\Lambda_0)^{\otimes k_0}\otimes L(\Lambda_j)^{\otimes k_j}$ generated by the highest weight vector $v_{L(\Lambda)}= v_{L(\Lambda_0)}^{\otimes k_0}\otimes v_{L(\Lambda_j)}^{\otimes k_j}$.
Suppose that
\begin{equation}\label{eq:d1DE67}
\sum_{a \in A}
c_{a}\, b^a v_{L(\Lambda)}=0, \quad\text{where}\quad c_a\in\CC,\, c_a\neq 0\text{ for all }a\in A,
\end{equation}
 $A$ is a finite nonempty set and all $b^a\in\BLL$ posses the same color-type. Let $b^{a_0}$ be a monomial of dual charge-type $\Dc$ such that $b^{a_0}< b^a$ for all $a\in A$, $a\neq a_0$, with respect to linear  order \eqref{order1}.  
 Applying the corresponding projection $\pi_{\Dc}$, which is defined in parallel with Section \ref{op}, on  linear combination \eqref{eq:d1DE67}, we obtain 
 \begin{equation}\label{eq:d2DE67}
\sum_{a \in A}
c_{a}\, \pi_{\Dc}\,b^a v_{L(\Lambda)}=0.  
\end{equation}
By Section \ref{voacon}, the highest weight vector $v_{L(\Lambda)}$ is identified with $1^{\ot k_0} \ot (e^{\lambda_j})^{\ot k_j}$, so that, due to \eqref{voaop3}, we have
$$
v_{L(\Lambda)}=1^{\ot k_0} \ot (e^{\lambda_j})^{\ot k_j}
=(1^{\ot k_0} \ot e_{\lambda_j}^{\ot k_j}) \,1^{\ot k}=(1^{\ot k_0} \ot e_{\lambda_j}^{\ot k_j}) v_{L(\Lambda_0)}^{\ot k}\quad\text{for }k=k_0+k_j.
$$
Therefore,
linear combination \eqref{eq:d2DE67} can be expressed as
$$
\sum_{a \in A}
c_{a}\, \pi_{\Dc}\,b^a (1^{\ot k_0} \ot e_{\lambda_j}^{\ot k_j}) v_{L(\Lambda_0)}^{\ot k}=0.  
$$
By employing \eqref{sv1} to move   $1^{\otimes k_0}\otimes e_{\lambda_j}^{\otimes k_j}$ all the way to the left and then  dropping the invertible  operator,  we get
$$
\sum_{a \in A}
c_{a}\, \pi_{\Dc}\, \tilde{b}^a  v_{L(\Lambda_0)}^{\otimes k } =0
$$
for some quasi-particle monomials $\tilde{b}^a$.
Using  the fact that the original monomials $b^a$ belong to $\BLL$ one can verify that
all $\tilde{b}^a$ belong to $\BL$. Therefore, due to the identification
$v_{L(\Lambda_0)}^{\otimes k } =v_{L(k\Lambda_0)}$, the linear independence of the set $\BLcc$ now follows from
Theorem \ref{thm_baza}.

We now proceed as in Section \ref{subsec55} and construct a spanning set for   $U(\ntplus)/I_{L(\Lambda)}$. 
We denote the image of the element $x\in U(\ntplus)$ in the quotient 
$U(\ntplus)/I_{V}$, where $V=L(k\Lambda_0),L(\Lambda)$, by   $\bar{x}$.
Let $\BLccbar$ be the set of all monomials
\beq\label{monomialbar}\tag{$\bar{m} $}
\bar{b}=
\Big(\bar{x}_{n_{r_{l}^{(1)},l}\alpha_{l}}(m_{r_{l}^{(1)},l})\ldots \bar{x}_{n_{1,l}\alpha_{l}}(m_{1,l})\Big)\ldots 
\Big( \bar{x}_{n_{r_{1}^{(1)},1}\alpha_{1}}(m_{r_{1}^{(1)},1})\ldots  \bar{x}_{n_{1,1}\alpha_{1}}(m_{1,1})\Big)
\eeq
in  $U(\ntplus)/I_{L(\Lambda)}$
such that
their charges and energies satisfy 
\beq\label{uvjeti}
n_{r_{i}^{(1)},i}\leqslant \ldots \leqslant n_{1,i}\fand 
m_{r_{i}^{(1)},i}\leqslant \ldots \leqslant m_{1,i}\qquad\text{for all } i=1,\ldots,l 
\eeq
and difference conditions \eqref{d1}, \eqref{d2DE67} and \eqref{d3}. It is clear from Theorem \ref{thm_baza} that the set of all monomials $\bar{b}$ as in \eqref{monomialbar} satisfying
\eqref{uvjeti} and difference conditions \eqref{d1}, \eqref{d2} and \eqref{d3} spans the quotient $U(\ntplus)/I_{L(\Lambda)}$. Suppose  that   such a monomial $\bar{b}$  does not satisfy the more restrictive condition \eqref{d2DE67}.
Introduce the generating functions
$$
\bar{X}_V=
 \bar{x}_{n_{r_{l}^{(1)},l}\alpha_{l}}(z_{r_{l}^{(1)},l}) \cdots 
 \bar{x}_{n_{2,1}\alpha_{1}}(z_{2,1})
 \bar{x}_{n_{1,1}\alpha_{1}}(z_{1,1})\quad\text{for }V=L(k\Lambda_0),L(\Lambda),
$$
where the subscript $V$ indicates that the coefficients of   $\bar{X}_V$ are regarded as elements of the quotient $U(\ntplus)/I_{V}$. 
Clearly, $\bar{b}$ equals the coefficient of the variables
$$
 z_{r_{l}^{(1)},l}^{-m_{r_{l}^{(1)},l}-n_{r_{l}^{(1)},l}}\cdots 
z_{2,1}^{-m_{2,1}-n_{2,1}}
z_{1,1}^{-m_{1,1}-n_{1,1}}
$$
in $\bar{X}_{L(\Lambda)}$.
By Theorem
\ref{thm_prezentacija} we have $W_{L(k\Lambda_0)}\cong U(\ntplus)/I_{L(k\Lambda_0)}$.
Therefore, due to commutation relations
$$
(z_{p,i}-z_{q,i'})^{ M_{i}}
x_{n_{q,i'}\alpha_{i'}}(z_{q,i'})
x_{n_{p,i}\alpha_{i}}(z_{p,i})
=(z_{p,i}-z_{q,i'})^{M_{i} }
x_{n_{p,i}\alpha_{i}}(z_{p,i})
x_{n_{q,i'}\alpha_{i'}}(z_{q,i'})  
$$
with $M_{i}=\min \left\{\textstyle  n_{q,i'},n_{p,i}\right\}$,
  the product $P \bar{X}_{L(k\Lambda_0)}$, where $P$ is the Laurent polynomial
$$
P=\prod_{i=2}^l
\prod_{q=1}^{r_{i'}^{(1)}}
\prod_{p=1}^{r_{i}^{(1)}}
\left(
1-\frac{z_{q,i'}}{z_{p,i}}
\right)^{\min
\left\{ n_{q,i'},n_{p,i}   
\right\}
},
$$
belongs to
\beq\label{equation123}
\prod_{i=1}^{l} \prod_{p=1}^{r_{i}^{(1)}} z_{p,i}^{-\sum_{q=1}^{r_{i'}^{(1)}}  \min\left\{n_{q,i'}, n_{p,i}\right\}}
 ( U(\ntplus)/I_{V})[[z_{r_{l}^{(1)},l},\ldots ,z_{1,1}]]
\eeq 
for $V=L(k\Lambda_0)$.
This implies that the product $P \bar{X}_{L(\Lambda)}$ belongs to
\eqref{equation123} for $V=L(\Lambda)$.
However, every vertex operator  $\bar{x}_{n\alpha_i}(z)$ in the product $P \bar{X}_{L(\Lambda)}$   can be moved all the way to the right. 
 By \eqref{idealDE} we have $x_{\alpha_j}(-1)^{k_0+1}\in I_{L(\Lambda)}$, so that each  $\bar{x}_{n\alpha_i}(z)$ increases the power of its variable $z$ in \eqref{equation123} by $\sum_{t=1}^n \delta_{i j_t}$. Therefore, we have
\beq\label{equation1234}
P \bar{X}_{L(\Lambda)}
\in
\prod_{i=1}^{l} \prod_{p=1}^{r_{i}^{(1)}} z_{p,i}^{\sum_{t=1}^{n_{p,i}} \delta_{i j_t}-\sum_{q=1}^{r_{i'}^{(1)}}  \min\left\{n_{q,i'}, n_{p,i}\right\}}
 ( U(\ntplus)/I_{L(\Lambda)})[[z_{r_{l}^{(1)},l},\ldots ,z_{1,1}]].
\eeq 
By employing \eqref{equation1234} and  repeating the corresponding part of the proof of \cite[Thm. 5.1]{G} the monomial $\bar{b}$ can be expressed as a linear combination of elements of 
 $\BLccbar$. Hence we conclude that the set $\BLccbar$ spans the quotient $U(\ntplus)/I_{L(\Lambda)}$.

Since the ideal $  I_{L(\Lambda)}$ belongs to the kernel of the map $f_{L(\Lambda)}$   defined by \eqref{map},
Theorems \ref{thm_baza_DE} and \ref{thm_prezentacija_DE} can be now verified by arguing as in Section \ref{subsec56}.

\subsection{Operator \texorpdfstring{$A_\theta$}{A-theta} revisited} \label{a}
As with type $G$ in \cite{B3}, the linear independence proof in type $F$ employs  certain operator $A_\theta = x_{\theta}(-1)$; see Sections \ref{subsec53} and   \ref{subsec54}.   
In this section we show that the operator $A_\theta$ associated with the   maximal root $\theta$  in type  $E$
 can be also used to verify the linear independence, but  of different  bases. 
First, for $\g=E_l$ set
$$
(i_1,\ldots ,i_l; i''_3,\ldots ,i''_l) =
\begin{cases}
(1,7,2,3,4,5,6,8; 1,2,3,4,5,5),&\text{if }l=8,\\
(1,6,5,4,3,2,7; 6,5,4,3,3),&\text{if }l=7,\\
(6,5,4,3,2,1; 5,4,3,2),&\text{if }l=6.
\end{cases}
$$
Introduce the following families of difference conditions:
\begin{align}
 & m_{p,i_j}\leqslant  -n_{p,i_j}
- 2(p-1) n_{p,i_j}\quad \text{for}\quad p=1,\ldots, r_{i_j}^{(1)}\fand j=1,2 ;\label{cd1}\tag{$c_2^{0}$}\\
&   m_{p,i_j}\leqslant  -n_{p,i_j}
+ \sum_{q=1}^{r_{i_j''}^{(1)}}\min\left\{n_{q,i_j''},n_{p,i_j}\right\}
- 2(p-1) n_{p,i_j}\quad \text{for}\quad  p=1,\ldots, r_{i_j}^{(1)}; \label{cd2}\tag{$c_2^{j}$}\\
&   m_{p,i_j}\leqslant  -n_{p,i_j}
+ \sum_{s=i''_j,i_k}\sum_{q=1}^{r_{s}^{(1)}}\min\left\{n_{q,s},n_{p,i_j}\right\}
- 2(p-1) n_{p,i_j}\quad
\text{for}\quad  p=1,\ldots, r_{i_j}^{(1)}.\label{cd3}\tag{$c_2^{j,k}$}
\end{align}
Let $\BL^{E_l}$ be the set   all monomials  \eqref{monomial} which satisfy 
\eqref{uvjeti} and 
the following difference conditions:
\begin{itemize}
\item[$\circ$] \eqref{d1}, \eqref{d3}, \eqref{cd1}, \eqref{cd2} for $j=3,4,5,6,8$ and \eqref{cd3} for $(j,k)=(7,2)$ if $l=8$;
\item[$\circ$] \eqref{d1}, \eqref{d3}, \eqref{cd1}, \eqref{cd2} for $j=3,4,5,7$ and \eqref{cd3} for $(j,k)=(6,1)$ if $l=7$;
\item[$\circ$] \eqref{d1}, \eqref{d3}, \eqref{cd1}, \eqref{cd2} for $j=3, 5,6$ and \eqref{cd3} for $(j,k)=(4,1)$ if $l=6$.
\end{itemize}

\begin{pro}\label{proposition}
For any positive integer $k$ the set 
$$\BLc^{E_l}=\left\{b v_{\LLL} : b\in \BL^{E_l}\right\}\subset W_{\LLL}$$
 forms a basis of the principal subspace $W_{\LLL}$ of the standard module $\LLL$ for the affine Lie algebra in type $E_l^{(1)}$.
\end{pro}

\begin{prf}
The maximal root $\theta $ in type $E$   satisfies
\beq\label{athetae}
\alpha_i(\theta^\vee) =\delta_{6i}\quad\text{for }\g=E_6
\Fand
\alpha_i(\theta^\vee) =\delta_{1i}\quad\text{for }\g=E_7,E_8.
\eeq
Therefore, as described in Section \ref{subsec54},  
by applying the procedure from Section \ref{subsec53} on an arbitrary linear combination of   elements of $\BLc^{E_8}$, 
 one can remove all quasi-particles of color $1$ from the corresponding quasi-particle monomials. 
The resulting linear combination can be  identified as a linear combination of elements of  $\BLc^{E_7}$; see Figure \ref{figure}. 
Due to \eqref{athetae}, by applying the same procedure once again,  one can remove  all quasi-particles of color $1$\footnote{
Note that the quasi-particles of color $1$ in type $E_7$ correspond, with respect to the aforementioned identification, to the quasi-particles of color $7$ in type $E_8$; see Figure \ref{figure}.
} from the corresponding quasi-particle monomials, thus obtaining  the  expression  which can be identified as a linear combination of elements of the basis $\BLc$ from Theorem \ref{thm_baza} for $\g=D_6$; see Figure \ref{figure}. 
As for type $E_6$, due to \eqref{athetae}, by applying the procedure from Section \ref{subsec53} on an arbitrary linear combination of   elements of $\BLc^{E_6}$, 
 one can remove all quasi-particles of color $6$ from the corresponding quasi-particle monomials. 
The resulting expression can be identified as a linear combination of elements of the basis $\BLc$ from Theorem \ref{thm_baza} for $\g=A_5$; see Figure \ref{figure}.
Therefore, the proposition follows from Theorem \ref{thm_baza} and the fact that the characters of the corresponding bases coincide which is verified by arguing as in Section \ref{sec40}. 
\end{prf}

\section{Character formulae and combinatorial identities}\label{sec40}

Let $\delta=\sum_{i=0}^l a_i \alpha_i$ be the imaginary root as in \cite[Chap. 5]{K}, where the integers $a_i$ denote the labels in the  Dynkin diagram \cite[Table Aff]{K} for $\gtilde$.
As before, let $V$ denote a standard module or a generalized Verma module.
Define the  character $\ch W_{V}$ of the corresponding principal subspace $W_V$ by 
 $$
 \ch W_{V}=\sum_{m,n_1,\ldots,n_l\geqslant 0} 
\dim (W_{V})_{-m\delta +n_1\alpha_1 +\ldots + n_l\alpha_l}\, q^{m}y^{n_1}_{1}\cdots y^{n_l}_{l},
$$
where $q,  y_1,\ldots, y_l$ are formal variables and $(W_{V})_{-m\delta +n_1\alpha_1 +\ldots + n_l\alpha_l}$ denote the  weight subspaces of $W_V$ of weight $-m\delta +n_1\alpha_1 +\cdots + n_l\alpha_l$ with respect to
$$\htilde =\mathfrak{h}\otimes \mathbb{C}[t,t^{-1}]\oplus \mathbb{C}c\oplus \mathbb{C}d.$$

In order to simplify  our notation, we set $\mu_i =\nu_i/\nu_{i'}$ for $i=2,\ldots ,l$; recall \eqref{niovi3}. Also, we write
$$(a;q)_r=\prod_{i=1}^r (1- aq^{i-1}) \ \ \text{for} \ \ r\geqslant 0
\Fand
(a; q)_{\infty} =\prod_{i\geqslant 1} (1- aq^{i-1}).
$$
Theorem \ref{thm_baza} implies the following character formulae:
\begin{thm}\label{thm_karakterL(kLambda0)}
{\normalfont(a)} 
Set  $n_i=\sum_{t=1}^{\nu_i k}r_i^{(t)}$  for $i=1,\ldots ,l$.
  For any integer $k\geqslant 1 $  we have
$$ \ch W_{L(k\Lambda_{0})}= 
\sum_{\substack{r_{1}^{(1)}\geqslant \cdots\geqslant r_{1}^{(\nu_1k)}\geqslant 0\vspace{-5pt}\\ \vdots\vspace{-2pt} \\r_{l}^{(1)}\geqslant \cdots\geqslant r_{l}^{(\nu_lk)}\geqslant 0}}
\frac{q^{\sum_{i=1}^l\sum_{t=1}^{\nu_{i} k}r_i^{(t)^2}
-\sum_{i=2}^{l}\sum_{t=1}^{k}
\sum_{p=0}^{\mu_i-1}r_{i'}^{(t)}
r_{i}^{\left(\mu_i t -p\right)} }} 
{\prod_{i=1}^{l}(q;q)_{r^{(1)}_{i}-r^{(2)}_{i}}\cdots (q;q)_{r^{(\nu_i k)}_{i}}}\,
\prod_{i=1}^{l}y^{n_i}_{i}.
$$
\noindent {\normalfont(b)}
Set $n_i=\sum_{t=1}^{\nu_iu_i}r_i^{(t)}$ for $i=1,\ldots ,l$.
  For any integer $k\geqslant 1 $  we have
\begin{eqnarray*}\label{karakterN(kLambda0)} 
&\ch  W_{N (k\Lambda_{0})} =\displaystyle \sum_{\substack{r_{1}^{(1)}\geqslant \cdots\geqslant r_{1}^{(\nu_1u_1)}\geqslant 0\vspace{-5pt}\\ \vdots\vspace{-2pt} \\ \substack{r_{l}^{(1)}\geqslant \cdots\geqslant r_{l}^{(\nu_lu_l)}\geqslant 
0\\  u_1, u_2, \ldots , u_l \geqslant 0}}}
 \frac{q^{\sum_{i=1}^l\sum_{t=1}^{\nu_{i} u_i}r_i^{(t)^2}-
\sum_{i=2}^{l}\sum_{t=1}^{u_i}\sum_{p=0}^{\mu_i-1}r_{i'}^{(t)}
r_{i}^{\left(\mu_i t -p\right)} }}
{\prod_{i=1}^{l}(q;q)_{r^{(1)}_{i}-r^{(2)}_{i}}\cdots (q;q)_{r^{(\nu_iu_i)}_{i}}}\prod_{i=1}^{l}y^{n_i}_{i}.  
\end{eqnarray*}
 \end{thm}
\begin{proof}
We  give the proof of this theorem for the case $F_4^{(1)}$, since the proof for the cases  $D_l^{(1)}$, $E_6^{(1)}$, $E_7^{(1)}$ and $E_8^{(1)}$ goes  analogously. The proof for other types can be found in \cite{G,B1,B2,B3}.
In order to determine the character of $W_{L(k\Lambda_{0})}$,   we write conditions on energies of quasi-particles of the set $B_{W_{L(k\Lambda_0)}}$  in terms of $r_i^{(s)}$. For a fixed color-type $(n_{4},n_{3},n_{2},n_{1})$, 
 charge-type
$$\quad\Cc=\left( n_{r_{4}^{(1)},4},\ldots , n_{1,4}; \,
n_{r_{3}^{(1)},3},\ldots , n_{1,3};\,
n_{r_{2}^{(1)},2},\ldots , n_{1,2};\,
 n_{r_{1}^{(1)},1},\ldots  , n_{1,1}\right)$$
and dual-charge-type
$$
 \Dc=\left(r^{(1)}_{4}, \ldots , r^{(2k)}_{4};\,
 r^{(1)}_{3}, \ldots , r^{(2k)}_{3};\,
 r^{(1)}_{2}, \ldots , r^{(k)}_{2}; \,
r^{(1)}_{1}, \ldots , r^{(k)}_{1}\right) 
$$
 the following equalities can be verified by a direct calculation:
\begin{equation}\label{uvjet1}
\sum_{p=1}^{r_{i}^{(1)}} (2(p-1)n_{p,i}+n_{p,i})= \sum_{t=1}^{k}r^{(t)^{2}}_{i}  \quad\text{for }i=1,2,\end{equation}
\begin{equation}\label{uvjet2}\sum_{p=1}^{r_{i}^{(1)}} ((2(p-1)n_{p,i}+n_{p,i})= \sum_{t=1}^{2k}r^{(t)^{2}}_{i}\quad\text{for }i=3,4,\end{equation}
\begin{equation}\label{uvjet3}
\sum_{p=1}^{r^{(1)}_{2}}\sum_{q=1}^{r^{(1)}_1}\mathrm{min}\{n_{p,2},n_{q,1}\}=\sum_{t=1}^{k}r^{(t)}_1r^{(t)}_2,\qquad
\sum_{p=1}^{r^{(1)}_{4}}\sum_{q=1}^{r^{(1)}_3}\mathrm{min}\{n_{p,4},n_{q,3}\}=\sum_{t=1}^{2k}r^{(t)}_3r^{(t)}_4,\end{equation}
\begin{equation}\label{uvjet5}\sum_{p=1}^{r^{(1)}_{3}}\sum_{q=1}^{r^{(1)}_2}\mathrm{min}\{n_{p,3},2n_{q,2}\}=\sum_{t=1}^{k}r^{(t)}_2(r^{(2t-1)}_3+r^{(2t)}_3).
\end{equation}
By combining \eqref{uvjet1}--\eqref{uvjet5}, difference conditions \eqref{d1}--\eqref{d3} and the formula
$$
\frac{1}{(q)_r}=\sum_{j\geqslant 0}p_r(j)q^j,
$$
where $p_r(j)$ denotes the number of partitions of $j$ with at most $r$ parts, we get
\begin{align}\nonumber 
\mathrm{ch} \  W_{L (k\Lambda_{0})} =& \sum_{\substack{r_{1}^{(1)}\geqslant \cdots\geqslant r_{1}^{(k)}\geqslant 0\\ r_{2}^{(1)}\geqslant \cdots\geqslant r_{2}^{(k)}\geqslant 0}}
 \frac{q^{\sum_{i=1}^{2}\sum_{t=1}^k r^{(t)^{2}}_{i}-\sum_{t=1}^kr^{(t)}_1r^{(t)}_{2}}}{\prod_{i=1}^{2}(q;q)_{r^{(1)}_{i}-r^{(2)}_{i}}\cdots (q;q)_{r^{(k)}_{i}}}\prod_{i=1}^{2}y^{n_i}_{i}\\
 \nonumber
& \times \sum_{\substack{r_{3}^{(1)}\geqslant \cdots\geqslant r_{3}^{(2k)}\geqslant 
0\\r^{(1)}_{4}\geqslant \ldots \geqslant r^{(2k)}_{4}\geqslant  0}}
 \frac{q^{\sum_{i=3}^{4}\sum_{t=1}^{2k}r^{(t)^{2}}_{i}-\sum_{t=1}^{2k}r^{(t)}_3r^{(t)}_{4}-\sum_{t=1}^kr_{2}^{(t)}(r_{3}^{(2t-1)}+r_{3}^{(2t)})}}{\prod_{i=3}^{4}(q;q)_{r^{(1)}_{i}-r^{(2)}_{i}}\cdots (q;q)_{r^{(2k)}_{i}}}\prod_{i=3}^{4}y^{n_i}_{i},& 
\end{align}
where $n_i=\sum_{t=1}^{k}r_i^{(t)}$ for $i=1,2$ and $n_i=\sum_{t=1}^{2k}r_i^{(t)}$  for $i=3,4$, as required. The character formula for the generalized Verma module is verified  analogously.
\end{proof}

Theorem \ref{thm_baza_DE} implies the following character formulae in types $D_l^{(1)}$, $E_6^{(1)}$ and $E_7^{(1)}$ while the case $A_l^{(1)}$ is due to \cite{G}.
\begin{thm}
Set  $n_i=r_i^{(1)}+\cdots +r_i^{(k)}$  for $i=1,\ldots ,l$.
  For any rectangular weight $\Lambda=k_0\Lambda_0+k_j\Lambda_j$ of level $k=k_0+k_j$   we have
$$ \ch W_{L(\Lambda)}= 
\sum_{\substack{r_{1}^{(1)}\geqslant \cdots\geqslant r_{1}^{(k)}\geqslant 0\vspace{-5pt}\\ \vdots\vspace{-2pt} \\r_{l}^{(1)}\geqslant \cdots\geqslant r_{l}^{(k)}\geqslant 0}}
\frac{q^{\sum_{i=1}^l\sum_{t=1}^{ k}r_i^{(t)^2}
-\sum_{i=2}^{l}\sum_{t=1}^{k}
r_{i'}^{(t)}
r_{i}^{\left( t \right)} +\sum_{i=1}^l\sum_{t=1}^k r_i^{(t)}\delta_{i j_t}}} 
{\prod_{i=1}^{l}(q;q)_{r^{(1)}_{i}-r^{(2)}_{i}}\cdots (q;q)_{r^{( k)}_{i}}}\,
\prod_{i=1}^{l}y^{n_i}_{i}.
$$
 \end{thm}

Note that from  \eqref{mapbar}  we have an isomorphism of $\ntplus$-modules $W_{N(k\Lambda_{0})}$ and $U(\ntplusleq)$, so we can   obtain character formula of $W_{N(k\Lambda_{0})}$ by using Poincar\'{e}--Birkhoff--Witt basis of $U(\ntplusleq)$ as well. For example, in the case $F_4^{(1)}$, we get
\begin{align}\label{t3:2} 
\ch  W_{N (k\Lambda_{0})} =\,\,&
\frac{1}{(qy_1, qy_1y_2,qy_1y_2y_3 , qy_1y_2 y_3 y_4,qy_2,qy_2y_3, qy_2y_3y_4,qy_2y_3^2; q)_{\infty}}\\
&\hspace{-60pt}\times \frac{1}{(qy_1y_2y_3^2, qy_1y_2y_3^2 y_4, qy_1y_2 y_3^2y_4^2,qy_1y_2^2y_3^2, qy_1y_2^2y_3^2y_4, qy_1y_2^2y_3^2y_4^2,qy_3,qy_2y_3^2y_4; q)_{\infty}}\non\\
 &\hspace{-60pt}\times \frac{1}{(qy_1y_2^2y_3^3y_4,qy_1y_2^2y_3^3y_4^2,qy_1y_2^2y_3^4y_4^2, qy_1y_2^3y_3^4y_4^2, qy_1^2y_2^3y_3^4y_4^2, qy_3y_4, qy_2y_3^2y_4^2, qy_4; q)_{\infty}}\non,
\end{align}
where
\begin{equation*}(a_1, \ldots, a_n; q)_{\infty} \coloneqq (a_1; q)_{\infty} \cdots (a_n; q)_{\infty}.\end{equation*}
For any  positive root $\alpha =a_1\alpha_1+\cdots a_l\alpha_l\in R_+$ we introduce the following notation
\begin{equation*}\label{oznakazaidentitet}
(\alpha; q)_{\infty}=(qy_1^{a_1},\ldots, qy_l^{a_l};q)_{\infty},
\end{equation*}
so that for an arbitrary affine Lie algebra $\gtilde$  character formula \eqref{t3:2} generalizes to
\beq\label{karakter}
\ch  W_{N (k\Lambda_{0})} = \frac{1} {\prod_{\alpha \in R_+}(\alpha; q)_{\infty} }.
\eeq
Theorem \ref{thm_karakterL(kLambda0)} and \eqref{karakter} imply the following generalization of Euler--Cauchy
theorem; cf. \cite{A}.
\begin{thm}\label{thm_identitet} 
For any untwisted  affine Lie algebra $\gtilde$ we have
\begin{eqnarray*}
   \displaystyle\frac{1}{\prod_{\alpha \in R_+}(\alpha; q)_{\infty} } = \displaystyle\sum_{\substack{r_{1}^{(1)}\geqslant \cdots \geqslant r_{1}^{(m)} \geqslant \cdots \geqslant 0\vspace{-5pt}\\ \hspace{-7pt}\vdots\vspace{-2pt} \\ \substack{r_{l}^{(1)}\geqslant \cdots \geqslant r_{l}^{(m)} \geqslant \cdots \geqslant 0}}}
 \frac{q^{\sum_{i=1}^l\sum_{t\geqslant 1} r_i^{(t)^2}-\sum_{i=2}^{l}\sum_{t\geqslant 1} \sum_{p=0}^{\mu_i-1}r_{i'}^{(t)}
r_{i}^{\left(\mu_i t -p\right)} }}{\prod_{i=1}^{l}\prod_{j\geqslant 1} (q;q)_{r^{(j)}_{i}-r^{(j+1)}_{i}}}\prod_{i=1}^{l}y^{n_i}_{i},
\end{eqnarray*}
where $n_i=\sum_{t\geqslant 1}r_i^{(t)}$ for $i=1,\ldots ,l$ and the sum on the right hand side of goes over all descending infinite sequences of nonnegative integers with finite support.
\end{thm}
In particular, the theorem produces three new families of combinatorial identities which correspond to types $D$, $E$ and $F$.

\section*{Acknowledgement}
The authors would like to thank Mirko Primc for useful discussions and support.
 The first author is partially supported by the QuantiXLie Centre of Excellence, a project cofinanced by the Croatian Government and European Union through the European Regional Development Fund - the Competitiveness and Cohesion Operational Programme  (Grant KK.01.1.1.01.0004).

\linespread{1.0}

\end{document}